\documentclass[12pt,a4paper,oneside]{article}
\usepackage{arxiv}

\usepackage[utf8]{inputenc}
\usepackage[numbers]{natbib}
\usepackage{graphicx}
\usepackage{subfigure}
\usepackage{placeins} 

\usepackage{url}

\usepackage{amsmath}
\usepackage{amsfonts}       
\usepackage{enumitem}
\usepackage{tabularx,booktabs} 
\usepackage{makecell}
\usepackage{multirow, makecell}

\usepackage{pdflscape}

\usepackage{wrapfig} 

\usepackage{tkz-tab}

\usepackage{amsthm}
\newtheorem{theorem}{Theorem}[section]
\newtheorem{corollary}[theorem]{Corollary}
\newtheorem{lemma}[theorem]{Lemma}

\newtheorem{remark}[theorem]{Remark}
\newtheorem{definition}{Definition}
\newtheorem{assumption}{Assumption}

\usepackage{ulem}
\usepackage{mathtools}

\usepackage{hyperref}

\newcommand{\ma}{\mathsf{S}}

\newcommand{\N}{\mathbb{N}}
\newcommand{\J}{\mathcal{J}}
\newcommand{\R}{\mathbb{R}}
\newcommand{\E}{\mathcal{E}}

\newcommand{\DT}{\Delta \theta}
\newcommand{\DP}{\Delta \psi}
\newcommand{\Om}{\mathsf{\Omega}}
\newcommand{\ovb}{\frac{1}{\beta}}
\newcommand{\dt}[1]{\frac{\diff #1}{\diff t}}

\newcommand{\tildet}{\widetilde\theta}
\newcommand{\tildep}{\widetilde\psi}

\newcommand*\diff{\mathop{}\!\mathrm{d}}

\newcommand{\lmin}{l_{\mathrm{min}}}
\newcommand{\lmax}{l_{\mathrm{max}}}

\title{Inertial Newton Algorithms Avoiding Strict Saddle Points}

\author{\textbf{Camille Castera}$^\ast$\\
	CNRS - IRIT\\
	Universit\'e de Toulouse\\
	France
}

\makeatletter

\begin{document}

	\maketitle

	\begin{abstract}
		We study the asymptotic behavior of second-order algorithms mixing Newton's method and inertial gradient descent in non-convex landscapes. We show that, despite the Newtonian behavior of these methods, they almost always escape strict saddle points. We also evidence the role played by the hyper-parameters of these methods in their qualitative behavior near critical points. The theoretical results are supported by numerical illustrations.
	\end{abstract}

	\renewcommand*{\thefootnote}{$^\ast$}
	\footnotetext[1]{Corresponding author: \texttt{camille.castera@protonmail.com}\\ \textit{Published in Journal of Optimization Theory and Applications 199(12):881--903}}
	\renewcommand*{\thefootnote}{\arabic{footnote}}
	\setcounter{footnote}{0} 


	\section{Introduction}
	Designing algorithms for large-scale optimization remains a major challenge and is crucial for modern machine learning problems. Many of these problems amount to the unconstrained minimization of a so-called \textit{loss function}:
	\begin{equation}\label{eq::genproblem}
		\min_{\theta\in\R^P} \J(\theta).
	\end{equation}
	Lots of efforts are put into building algorithms exploiting second-order derivatives of $\J$ while maintaining low storage and computational costs.
	To this aim, a popular line of work consists in deriving algorithms from ordinary differential equations (ODEs), such as the following ODE \citep{alvarez2002second}:
		\begin{equation}\label{eq::2ndorderDIN}
			\frac{\diff^2{\theta}}{\diff t^2}(t) + \alpha\frac{\diff{\theta}}{\diff t}(t) + \beta \nabla^2\J(\theta(t))\frac{\diff{\theta}}{\diff t}(t) + \nabla\J(\theta(t))=0, \quad \text{for all } t> 0,
		\end{equation}
		where $\nabla\J$ and $\nabla^2\J$ denote the gradient and the Hessian of $\J$ respectively, $\alpha$ and $\beta$ are non-negative fixed parameters, and the rest of the setting is made precise later. Building algorithms from \eqref{eq::2ndorderDIN} is relevant for tackling \eqref{eq::genproblem}, indeed, if a solution of \eqref{eq::2ndorderDIN} converges, its limit is a critical point of $\J$ \cite{alvarez2002second}.
		This ODE is called DIN for \textit{dynamical inertial Newton-like system}, and echos famous optimization algorithms: taking $\beta=0$, \eqref{eq::2ndorderDIN} boils down to the heavy-ball with friction (HBF) ODE \citep{polyak1964some} and can be extended to connect it to Nesterov's method \citep{nesterov1983method,su2014differential,attouch2016fast},
		while when taking $\alpha=0$, \eqref{eq::2ndorderDIN} represents an inertial Newton method \citep{attouch2001second}.
		The term involving $\nabla^2\J$ in DIN provides stability and reduced oscillations to its solutions compared to HBF or Nesterov's method, which allows fast vanishing of the gradient \cite{attouch2016fast}.
		DIN is thus at the interface between first and second order optimization, yet unlike most second-order dynamics, it has the notable property to possess an equivalent form where $\nabla^2\J$ does not appear explicitly, see \eqref{eq::DIN} below. By discretizing this formulation, \citet{castera2019inertial} recently built an algorithm, called INNA. This algorithm has the same cost as first-order algorithms (only requires evaluations of $\nabla\J$), and is naturally extendable to stochastic and non-smooth settings because it does not rely on ``Hessian-vector'' products, unlike most cheap second-order methods.

		While being of second-order type, DIN-like dynamics and INNA do not yield faster rates of convergence (in values) for convex function than the optimal first-order methods (see \textit{e.g.}, \cite{attouch2016fast}). Yet INNA revealed to perform well in practical problems \textit{e.g.}, for training neural networks, and featured good ``generalization performances'', see \cite{castera2019inertial}. Therefore, INNA has mostly proved to be useful to minimize\footnote{The limit of sub-sequences of iterates of INNA yield critical points of $\J$, both for vanishing step-sizes \citep{castera2019inertial}, and fixed ones if $\J$ has Lipschitz continuous gradient (see Theorem~\ref{thm::convergenceofINNA}).} non-convex functions (as is the case in deep learning). Yet non-convex functions may possess spurious critical points (\textit{i.e.}, that are not minima). This is problematic because, while first-order methods are likely to avoid \textit{strict} saddles (critical points where $\nabla^2\J$ has a negative eigenvalue, \cite{goudou2009gradient,lee2016gradient,o2019behavior}), vanilla Newton's method is  attracted to any type of critical points (see \textit{e.g.}, \cite{dauphin2014}). Since INNA and DIN are at the interface between first and second order methods, it remains open to know \textit{whether the solutions of DIN and INNA are likely to avoid strict saddle points?}
		The main contribution of this paper is to answer positively to this question both for DIN and INNA with fixed step-sizes, for any choice of parameters $\alpha>0$ and $\beta>0$. We also shed light on the influence of $\alpha$ and $\beta$ on the asymptotic behavior of the solutions of DIN, this gives new insights into the role played by these parameters.

	To summarize, our main contributions are the following:
	\begin{itemize}
		\setlength\itemsep{0.1em}
		\item[--] Proving that the solutions of DIN almost always avoid strict saddle points.
		\item[--] Studying the convergence of INNA with fixed step-sizes and prove again almost sure avoidance of strict saddles.
		\item[--] Showing new results on the qualitative behavior of the solutions of DIN and on the role played by the hyper-parameters $\alpha$ and $\beta$.
	\end{itemize}

	\paragraph{Related work.}
	ODEs (or dynamical systems) are a powerful tool to design first-order methods  \cite{polyak1964some} and to provide new understanding of existing algorithms  \citep{su2014differential,shi2018understanding}. They allow deriving rates of convergence both for convex \citep{attouch2018fast,vassilis2018differential,attouch2019rate,aujol2019optimal} and non-convex\footnote{This requires more assumptions on $\J$, \textit{e.g.}, the Kurdyka-{\L}ojasiewicz property, see  \citep{laszlo2021convergence}.} \citep{laszlo2021convergence} functions. First-order algorithms can even be built from dynamical systems in non-convex and non-smooth settings \citep{boct2016inertial,ochs2018local,alecsa2020gradient}.

	The DIN \cite{alvarez2002second} ODE was studied by many, among which \cite{attouch2014dynamical,attouch2016fast} who considered extensions of DIN where the hyper-parameters $\alpha$ and $\beta$ vary over time; connections between the Nesterov's method and DIN-like ODEs \citep{shi2018understanding,alecsa2021extension} were made;
	generalizations and extensions have been considered, including Tikonov regularization \citep{boct2021tikhonov}, closed-loop dampings \citep{attouch2022fast}, and non-smooth \citep{attouch2020newton,attouch2021continuous} and non-convex settings \citep{castera2019inertial}.
	The first-order equivalent formulation of DIN was exploited to design several algorithms \cite{chen2019first,attouch2019first}, including INNA \citep{castera2019inertial}.
	The influence of the parameters $\alpha$ and $\beta$ on rates of convergence  was studied in the convex and strongly-convex settings by \cite{attouch2019first}. We consider non-convex functions and rather provide qualitative properties (such as a spiraling phenomenon in Section~\ref{sec::hartman}). This analysis relies on the Hartman-Grobman theorem \citep{grobman1959homeomorphism,hartman1960lemma}.

	Our main results rely on the theory of dynamical systems, and in particular on the stable manifold theorem \citep{pliss1964reduction,kelley1966stable}. It allows showing that algorithms are likely to avoid strict saddle points, and was used first for gradient descent (GD) and HBF \cite{goudou2009gradient,lee2016gradient,o2019behavior} on functions with isolated critical points. Results have then been extended to non-isolated critical points \citep{panageas2017}, to GD with non-constant step-sizes \citep{panageas2019first,truong2019convergence} and to SGD \citep{mertikopoulos2020almost}.
	Finally, the convergence of INNA was proved for vanishing step-sizes \cite{castera2019inertial}, we prove it for fixed step-sizes.

	\paragraph{Organization.} We specify the setting in Section~\ref{sec::preliminary}. Section~\ref{sec::DIN} states the results for DIN, in particular the avoidance of strict saddles (Section~\ref{sec::stabmanif}), a qualitative study of DIN is carried out in Section~\ref{sec::hartman}. The reader only interested in results for INNA may go to Section~\ref{sec::INNA} where convergence and avoidance of saddles are proved and experiments are presented. Conclusions are finally drawn.

	\section{Preliminary Discussions and Definitions}\label{sec::preliminary}
	In the rest of the paper, we fix $P\in\N_{>0}$, the dimension of the problem \eqref{eq::genproblem}, and equip $\R^P$ with norm $\Vert\cdot\Vert$ and scalar product $\langle\cdot,\cdot\rangle$. We consider a $C^2$ lower-bounded loss function $\J\colon\R^P\to\R$ and denote by $\nabla\J$ and $\nabla^2\J$ its gradient and Hessian, respectively. We say that $\theta:\R_+\to\R^P$ is a solution or trajectory of \eqref{eq::2ndorderDIN} if it is $C^2(\R_+)$ and solves \eqref{eq::2ndorderDIN} for all $t>0$. We say that $\theta$ converges if $\lim_{t\to+\infty}\theta(t)$ exists. Finally, we fix two constants $\alpha\geq 0$ and $\beta>0$.

	We recall optimality conditions, see \cite{nocedal2006numerical}:
	if $\theta^\star$ is a minimizer of $\J$, then
		$\theta^\star$ is a critical point of $\J$ (\textit{i.e.}\ $\nabla \J(\theta^\star)=0$) and $\nabla^2 \J(\theta^\star)$ is positive semidefinite (its eigenvalues are \emph{non-negative}).
		We distinguish three types of critical points $\theta^\star\in\R^P$, depending on the eigenvalues of $\nabla^2\J(\theta^\star)$. If $\nabla^2\J(\theta^\star)$ has:
		\\
		--\ only positive eigenvalues, then $\theta^\star$ is a (local) minimizer;
		\\
		--\  at least one negative eigenvalue, $\theta^\star$ is called \textit{strict saddle point}. It cannot be a minimizer, but may (not necessarily) be a maximizer;
		\\
		--\ only non-negative eigenvalues and at least one zero eigenvalue, $\theta^\star$ is a \textit{non-strict} saddle point. It may be maximizer, a minimizer, or neither of them, see the example in Figure~\ref{fig::saddles}.

	Due to the difficulties raised by the existence of non-strict saddle points, we will sometimes consider Morse functions\footnote{The main algorithmic result of this paper, Theorem~\ref{thm::MainResINNA} holds beyond Morse functions.}, defined next.
	\begin{definition}
		We say that  $\J$ is a \textit{Morse function} if $\nabla^2 \J(\theta^\star)$ has no zero eigenvalues at critical points $\theta^\star\in\R^P$.
	\end{definition}
	Critical points of Morse functions may only be strict saddles or minima and they are isolated: they are the only critical point in a neighborhood. This is because $\nabla\J$ cannot be constant around $\theta^\star\in\R^P$ if $\nabla^2\J(\theta^\star)$ has only non-zero eigenvalues, see \cite[Corollary 2.3]{milnor2016morse}. We now move on to the analysis of DIN.
	\begin{figure}[t]
		\begin{minipage}{0.49\linewidth}
			\centering
			\includegraphics[width=.7\linewidth]{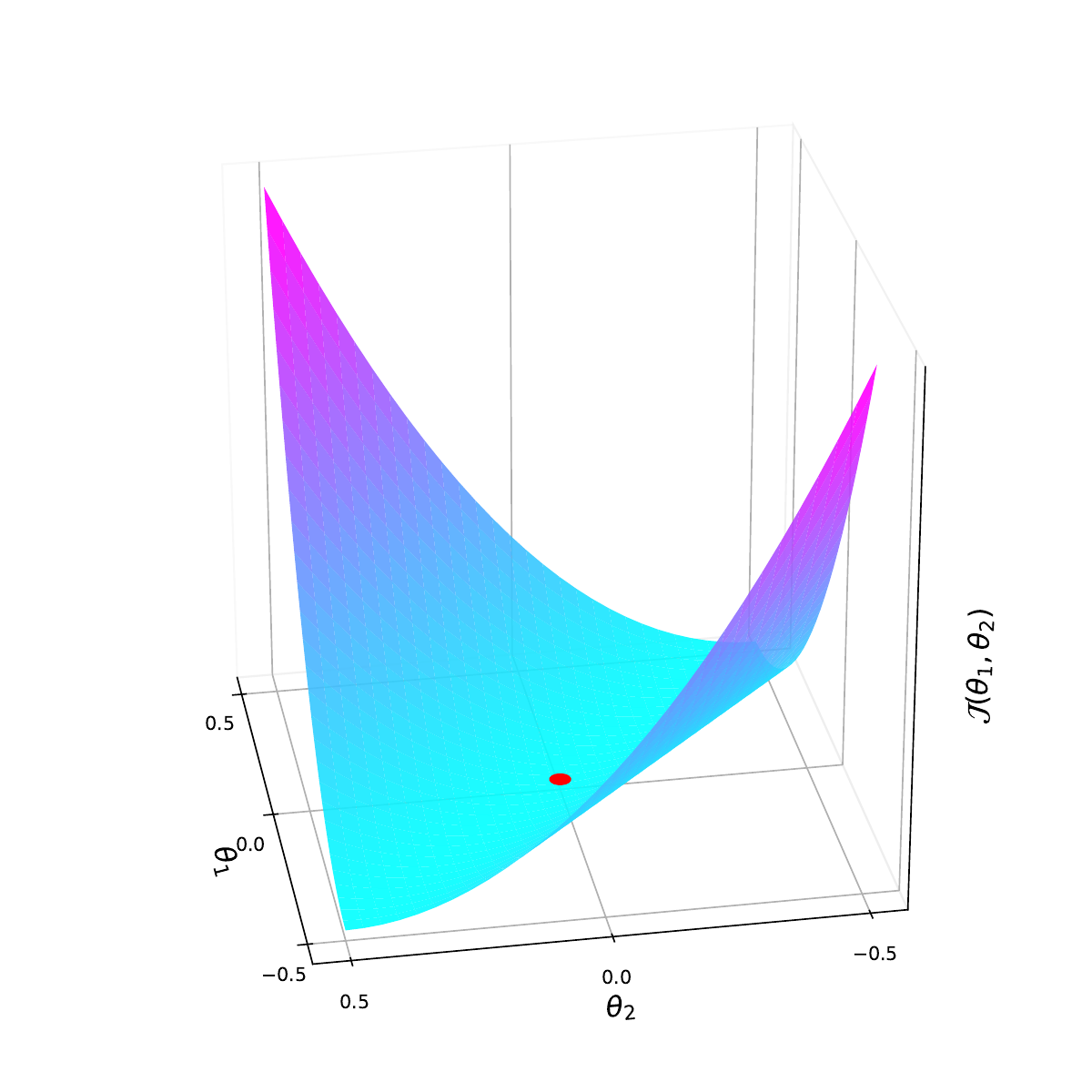}
		\end{minipage}
		\begin{minipage}{0.49\linewidth}
			\centering
			\includegraphics[width=.7\linewidth]{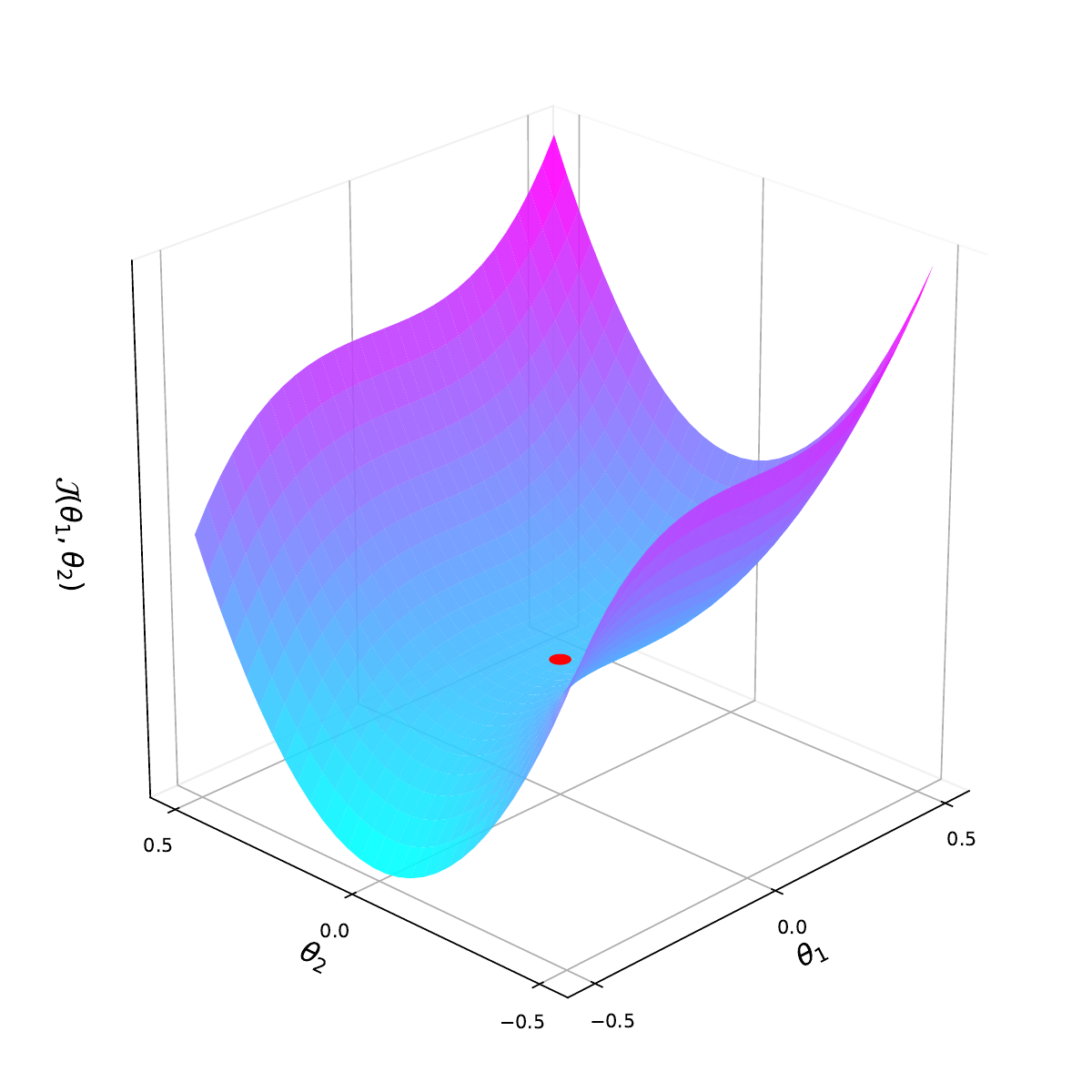}
		\end{minipage}
		\caption{Example of two functions with non-strict saddle. On the left figure, $(0,0)$ is a minimum, but on the right, the critical point  $(0,0)$ is neither a minimum nor a maximum. \label{fig::saddles}}
	\end{figure}
	\section{Asymptotic Behavior of the Solutions of DIN}\label{sec::DIN}
	As mentioned in the introduction, a powerful property of \eqref{eq::2ndorderDIN} is that it is equivalent to the following first-order system \citep{alvarez2002second}:
	\begin{equation}\label{eq::DIN}
		\begin{cases}
			\dt{\theta}(t)&= -(\alpha-\ovb)\theta(t) - \ovb \psi(t) -\beta\nabla\J(\theta(t))   \\
			\dt{\psi}(t)&= -(\alpha-\ovb)\theta(t) - \ovb \psi(t)
		\end{cases}, \quad \text{for all } t> 0,
	\end{equation}
	where $(\theta,\psi)\colon\R_+\times\R_+\to\R^P\times\R^P$.
	Since $\J$ is $C^2(\R^P)$, the existence and uniqueness (with respect to initial conditions) of  solutions of \eqref{eq::DIN} follow from the Cauchy-Lipschitz theorem \citep{alvarez2002second}. We now study stationary points of \eqref{eq::DIN}.

	\paragraph{Characterization of stationary points.}
	We say that $(\theta^\star,\psi^\star)\in\R^P\times\R^P$ is a \emph{stationary point} for \eqref{eq::DIN} if for any solution $(\theta,\psi)$ of \eqref{eq::DIN}, whenever there exists $t_0\geq0$ such that $(\theta(t_0),\psi(t_0)) = (\theta^\star,\psi^\star)$, then $\forall t\geq t_0$, $(\theta(t),\psi(t)) = (\theta^\star,\psi^\star)$, or equivalently, $\forall t\geq t_0$, $\dt\theta(t)=\dt\psi(t) = 0$. From \eqref{eq::DIN}, stationary points are then those such that $\nabla\J(\theta^\star)=0$, and $\psi^\star = (1-\alpha\beta)\theta^\star$, that is, the set
	$
	\ma  = \left\{(\theta^\star,\psi^\star)\in\R^P\times\R^P\middle| \nabla\J(\theta^\star)=0,\ \psi^\star=(1-\alpha\beta)\theta^\star\right\}.
	$
	Remark that there is a one-to-one correspondence between stationary points of \eqref{eq::DIN} and critical points of $\J$, so as previously mentioned, the limit (if it exists) of a solution $(\theta,\psi)$ of \eqref{eq::DIN} is in $\ma$ \citep{alvarez2002second}, meaning that $\theta$ converges to a critical point.

	Existence of a limit is not always guaranteed for the general class of smooth non-convex functions, but still holds for broad classes of problems. In particular, when $\J$ is analytic or semi-algebraic \cite{alvarez2002second}, or more generally when it possesses the Kurdyka-{\L}ojasiewicz\footnote{See \textit{e.g.}, \cite{alvarez2002second,castera2019inertial} for precise definitions. These notions are not crucial in what follows.} (KL) property, then bounded solutions of \eqref{eq::DIN} converge to $\ma$ \cite[Theorem 13]{castera2019inertial}.
	In the sequel, we will study more specifically what type of stationary points the solutions of \eqref{eq::DIN} are likely to converge to, and then study the qualitative asymptotic behavior of these solutions.

	\subsection{DIN is Likely to Avoid Strict Saddle Points}\label{sec::stabmanif}
	We start with our main result regarding the limit of the solutions of DIN.
	\subsubsection{Main Convergence Results}
	We define the set of stationary points $(\theta^\star,\psi^\star)$ where  $\theta^\star$ is a strict saddle of $\J$:
	\begin{equation*}
		\ma_{<0} = \left\{(\theta^\star,\psi^\star)\in\ma \,\middle|  \nabla^2\J(\theta^\star)\text{ has at least one negative eigenvalue}\right\}.
	\end{equation*}
	\begin{theorem}\label{thm::MainResDIN}
		Assume that $\J$ is a Morse function, then for almost any initialization, the corresponding solution of \eqref{eq::DIN} does not converge to a point in $\ma_{<0}$.
	\end{theorem}
	Before proving the theorem, the following corollary is a straightforward consequence suited for practical applications.
	\begin{corollary}\label{cor::dincor}
		Assume that $\J$ is a Morse function, is coercive (\textit{i.e.}, $\Vert\theta\Vert\to+\infty \implies \J(\theta)\to +\infty$), and that bounded solutions of \eqref{eq::DIN} converge.
		If the initialization $(\theta_0,\psi_0)$ is random\footnote{The distribution of $(\theta_0,\psi_0)$ must be absolutely continuous w.r.t. the Lebesgue measure, that is: for any set $\mathsf{I}\subset \R^P\times\R^P$ with zero Lebesgue measure, $\mathbb{P}((\theta_0,\psi_0)\in\mathsf{I})=0$.}, then the corresponding solution $(\theta,\psi)$ of \eqref{eq::DIN} converges and the limit of $\theta$ is almost surely a local minimizer of $\J$.
	\end{corollary}
	\begin{proof}[Proof of Corollary~\ref{cor::dincor}]
		From \cite[Section 3.2]{castera2019inertial}, the coercivity of $\J$ ensures that any solution of \eqref{eq::DIN} remains bounded, so by assumption for any initialization the solution of \eqref{eq::DIN} converges, and its limit is in $\ma$.
		Let $(\theta_0,\psi_0)$ be a random variable with absolutely continuous distribution w.r.t. the Lebesgue measure on $\R^P\times\R^P$.
		According to Theorem~\ref{thm::MainResDIN}, the set $\mathsf{I}_{<0}$ of initializations such that the corresponding  solution $(\theta,\psi)$ of \eqref{eq::DIN} converges to $\ma_{<0}$ has zero measure, so by absolute continuity of the distribution, the probability to sample $(\theta_0,\psi_0)$ from $\mathsf{I}_{<0}$ is zero.
		Thus, with probability one, $(\theta,\psi)$ converges to $\ma \setminus\ma_{<0}$, which contain only minimizers of $\J$ since $\J$ is a Morse function.
	\end{proof}
	Note that the coercivity assumption in Corollary~\ref{cor::dincor} is only used to ensure the boundedness of solutions of \eqref{eq::DIN}.
	We now introduce the main tool to prove Theorem~\ref{thm::MainResDIN}: the stable manifold theorem.

	\subsubsection{The Stable Manifold Theorem}
	To simplify the notations we introduce the following mapping:
	$$G:(\theta,\psi)\in\R^P\times\R^P\mapsto \begin{pmatrix}
		-(\alpha-\ovb)\theta - \ovb \psi -\beta\nabla\J(\theta) \\
		-(\alpha-\ovb)\theta - \ovb \psi
	\end{pmatrix},$$ so that \eqref{eq::DIN} can be rewritten as
	\begin{equation}\label{eq::systemG}
		\frac{\diff}{\diff t}\begin{pmatrix}\theta(t)\\\psi(t)\end{pmatrix} = G(\theta(t),\psi(t)), \quad \text{for all } t> 0.
	\end{equation}
	For any $(\theta,\psi)\in\R^P\times\R^P$, we also denote by $DG(\theta,\psi)\in\R^{2P\times 2P}$ the Jacobian matrix of $G$ at $(\theta,\psi)$.
	Remark that $(\theta,\psi)\in\ma  \iff G(\theta,\psi)=0$, so the stationary points of \eqref{eq::DIN} are exactly the zeros of $G$.
	We now state the stable manifold theorem which is the keystone for proving Theorem~\ref{thm::MainResDIN}.

	\begin{theorem}[{\citep[Section~2.7, Page 107]{perko2013differential}}]\label{thm::stabman}
		Let $F\colon \R^{2P}\to\R^{2P}$ be $C^1$, denote by $DF$ its Jacobian. Consider the autonomous ODE,
		\begin{equation}\label{eq::autonomousODE}
			\dt{\Theta}(t) = F(\Theta(t)), \quad \text{for all } t> 0.
		\end{equation}
		Let $\Theta^\star\in\R^{2P}$ such that $F(\Theta^\star) = 0$ and such that $DF(\Theta^\star)$ has no eigenvalues with zero real part.  Assume that $DF(\Theta^\star)$ has $k\in\N$ eigenvalues with negative real part and let $E^{s}_{\Theta^\star}$ be the linear subspace spanned by these eigenvalues.
		Then, there exists a neighborhood $\Om_{\Theta^\star}$ of $\Theta^\star$ and a $C^1$ manifold $\mathsf{W}^{s}_{\Theta^\star}$ tangent to $E^{s}_{\Theta^\star}$ at $\Theta^\star$, whose dimension is $k$, such that, for any $\Theta_0\in \Om_{\Theta^\star}$ and its corresponding solution $\Theta$ of \eqref{eq::autonomousODE} it holds that:
		\renewcommand{\theenumi}{(\roman{enumi})}%
		\begin{enumerate}
			\item  If $\Theta_0\in \mathsf{W}^{s}_{\Theta^\star}$ then for all $t\geq 0$, $\Theta(t)\in\mathsf{W}^{s}_{\Theta^\star}$. \label{point1}
			\item $\left[\forall t\geq 0,\ \Theta(t)\in \Om_{\Theta^\star}\text{ and } \lim_{t\to+\infty} \Theta(t) = \Theta^\star\right]$ $\iff  \Theta_0\in \mathsf{W}^{s}_{\Theta^\star}$. \label{point2}
		\end{enumerate}%
	\end{theorem}

	\begin{remark}\label{rem::remonstabman}
		In \cite{perko2013differential}, \ref{point2} is stated outside of the statement of the theorem (see Page~114 therein).
	\end{remark}

	The manifold $\mathsf{W}^{s}_{\Theta^\star}$ is called the \emph{stable} manifold (hence the superscript $s$) because Theorem~\ref{thm::stabman} implies that all the solutions of \eqref{eq::autonomousODE} converging to $\Theta^\star$ must stay inside $\mathsf{W}^{s}_{\Theta^\star}\cap\Om_{\Theta^\star}$ after some time.
	We see why the proof of Theorem~\ref{thm::MainResDIN} relies on it. In particular, from Remark~\ref{rem::remonstabman}, if $DF(\Theta^\star)$ has one or more eigenvalues with positive real part, then $\mathsf{W}^{s}_{\Theta^\star}$ has zero measure. We will prove that this holds for $G$ at any point in $\ma_{<0}$.

	\subsubsection{Proof of Theorem~\ref{thm::MainResDIN}}

	\begin{proof}[Proof of Theorem~\ref{thm::MainResDIN}]
		Let $(\theta^\star,\psi^\star)\in\ma_{<0}$, the Jacobian of $G$ at $(\theta^\star,\psi^\star)$ is
		$DG(\theta^\star,\psi^\star) =
			\begin{pmatrix}
				-\beta\nabla^2\J(\theta^\star) - (\alpha-\ovb) I_P && -\ovb I_P\\
				-(\alpha-\ovb) I_P && -\ovb I_P
			\end{pmatrix},
		$ where $I_P$ denotes the identity matrix of $\R^{P\times P}$, and $DG(\theta^\star,\psi^\star)$ is displayed in four blocks.
		In order to apply Theorem~\ref{thm::stabman} to \eqref{eq::systemG}, we want to show that since $\nabla^2\J(\theta^\star)$ has a negative eigenvalue then $DG(\theta^\star,\psi^\star)$ has at least one eigenvalue with positive real part, and thus according to Theorem~\ref{thm::stabman}, the stable manifold associated to  $(\theta^\star,\psi^\star)$ has zero measure.
		First, $\nabla^2\J(\theta^\star)$ is real and symmetric, so there exists an orthogonal matrix $V$ such that $V^T \nabla^2\J(\theta^\star) V$ is diagonal, so
		\begin{equation}\label{eq::3diagMat}
			\begin{pmatrix}
				V^T & 0\\ 0 & V^T
			\end{pmatrix} DG(\theta^\star,\psi^\star) \begin{pmatrix}
				V & 0\\ 0 & V
			\end{pmatrix} =
			\begin{pmatrix}
				-\beta V^T\nabla^2\J(\theta^\star)V - (\alpha-\ovb) I_P && -\ovb I_P\\
				-(\alpha-\ovb) I_P && -\ovb I_P
			\end{pmatrix}
		\end{equation}
		is a matrix with only $3$ non-zero diagonals and whose eigenvalues are the same as those of $DG(\theta^\star,\psi^\star)$. Exploiting the tridiagonal structure, there exists a symmetric permutation $U\in\R^{2P\times 2P}$---specified in \eqref{eq::permutationU} in Appendix~\ref{sec::Permut}---such that we can transform \eqref{eq::3diagMat} into a block diagonal matrix:
		\begin{equation}\label{eq::blockdiagMat}
			U^T\begin{pmatrix}
				V^T & 0\\ 0 & V^T
			\end{pmatrix} DG(\theta^\star,\psi^\star) \begin{pmatrix}
				V & 0\\ 0 & V
			\end{pmatrix}U =
			\begin{pmatrix}
				M_1 & & \\
				&\ddots&\\
				& & M_P
			\end{pmatrix},
		\end{equation}
		where, for each $p\in\{1,\ldots,P\}$, $M_p= \begin{pmatrix}
			-(\alpha-\ovb) - \beta\lambda_p & -\ovb\\
			-(\alpha-\ovb) & -\ovb
		\end{pmatrix}$, up to a symmetric permutation, and where $\lambda_p$ is a corresponding eigenvalue of $\nabla^2\J(\theta^\star)$.

		The eigenvalues of $DG(\theta^\star,\psi^\star)$ are those of the matrices $M_p$ and are given for each $p\in\{1,\ldots,P\}$ by the roots of its characteristic polynomial: $\chi_{M_p}(X)= X^2 - \mathrm{trace}(M_p)X + \mathrm{det}(M_p)$, which yields,
		\begin{equation} \label{eq::chiMp}
			\chi_{M_p}(X) = X^2 + (\alpha + \beta\lambda_p)X + \lambda_p.
		\end{equation}
		This is a second-order polynomial and remark that $\chi_{M_p}(0)=\lambda_p$. Since $\J$ is a Morse function $\lambda_p\neq 0$, therefore $0$ is never a root of $\chi_{M_p}$. Now, either $\lambda_p<0$, then $\chi_{M_p}(0)<0$ and $\lim_{X\to+\infty}\chi_{M_p}(X) = +\infty$ so by continuity of $\chi_{M_p}$ its roots are real and at least one is positive. Or, $\lambda_p>0$, in that case $\alpha+\beta\lambda_p>0$ so there cannot be roots of $\chi_{M_p}$ with zero real part (see also Lemma~\ref{lem::lambdappositive} hereafter). To summarize, for any $(\theta^\star,\psi^\star)\in\ma_{<0}$, $DG(\theta^\star,\psi^\star)$ has no eigenvalues with zero real part and at least one positive eigenvalue, so we can now use Theorem~\ref{thm::stabman}.

		Let $(\theta^\star,\psi^\star)\in\ma_{<0}$, and consider the neighborhood $\Om_{(\theta^\star,\psi^\star)}$ and the manifold $\mathsf{W}^{s}_{(\theta^\star,\psi^\star)}$ containing $(\theta^\star,\psi^\star)$ stated in Theorem~\ref{thm::stabman}.
		Let an initialization $(\theta_0,\psi_0)\in\R^P\times\R^P$ such that the corresponding solution $(\theta,\psi)$ of \eqref{eq::systemG} converges to $(\theta^\star,\psi^\star)$.
		Consider $\Phi:\R^P\times\R^P\times \R\to \R^P\times\R^P$, the flow of the solutions of \eqref{eq::systemG}, so that we have in particular for all $t\geq 0$, $(\theta(t),\psi(t))= \Phi((\theta_0,\psi_0),t)$ and $(\theta_0,\psi_0)= \Phi((\theta(t),\psi(t)),-t)$.
		Since $(\theta,\psi)$ converges to $(\theta^\star,\psi^\star)$, there exists $t_0\geq 0$, such that for all $t\geq t_0$, $\Phi((\theta_0,\psi_0),t)\in\Om_{(\theta^\star,\psi^\star)}$.
		According to \ref{point2} in Theorem~\ref{thm::stabman} this implies that $\Phi((\theta_0,\psi_0),t_0)\in\mathsf{W}^{s}_{(\theta^\star,\psi^\star)}$, and so from \ref{point1} $\forall t\geq t_0$, $\Phi((\theta_0,\psi_0),t)\in \mathsf{W}^{s}_{(\theta^\star,\psi^\star)}$. Reversing the flow, we obtain, that $\forall t\geq t_0$, $(\theta_0,\psi_0)\in \Phi(\Om_{(\theta^\star,\psi^\star)}\cap\mathsf{W}^{s}_{(\theta^\star,\psi^\star)},-t)$, or more generally,
		\begin{equation}\label{eq::Flows}
			(\theta_0,\psi_0) \in \bigcup_{k\in\N}\Phi\left(\Om_{(\theta^\star,\psi^\star)}\cap\mathsf{W}^{s}_{(\theta^\star,\psi^\star)},-k\right),
		\end{equation}
		where the right-hand side is the union over $k\in\N$ of initial conditions such that the corresponding solution has reached $\Om_{(\theta^\star,\psi^\star)}\cap\mathsf{W}^{s}_{(\theta^\star,\psi^\star)}$ at time $k$.

		The set defined in the right-hand side of \eqref{eq::Flows} has zero measure. Indeed, we showed that $(\theta^\star,\psi^\star)\in\ma_{<0}$ implies that $DG(\theta^\star,\psi^\star)$ has one eigenvalue with positive real part, so by Theorem~\ref{thm::stabman}, $\mathsf{W}^{s}_{(\theta^\star,\psi^\star)}$ has dimension strictly less than $2P$ (the dimension of $\R^P\times\R^P$), hence this manifold has zero measure. Due to the uniqueness of the solution of \eqref{eq::DIN}, for each $k\in\N$, $\Phi(\cdot,-k)$ is a local diffeomorphism, so $\Phi\left(\Om_{(\theta^\star,\psi^\star)}\cap\mathsf{W}^{s}_{(\theta^\star,\psi^\star)},-k\right)$ is also a zero-measure set. Finally, a countable union of zero-measure sets also has zero measure. So we proved that the set
		\begin{equation*}
			 \mathsf{I}_{(\theta^\star,\psi^\star)}=\left\{(\theta_0,\psi_0)\in\R^P\times\R^P\middle| \Phi((\theta_0,\psi_0),t)\xrightarrow[t\to+\infty]{} (\theta^\star,\psi^\star)\right\},
		\end{equation*}
		of initial conditions such that the associated solutions converge to $(\theta^\star,\psi^\star)$ has measure zero.

		Finally, since $\J$ is a Morse function it has isolated  critical points so the set $\ma_{<0}$ is countable and applying the reasoning above to any $(\theta^\star,\psi^\star)\in\ma_{<0}$, we get that   $\bigcup_{(\theta^\star,\psi^\star)\in\ma_{<0}} \mathsf{I}_{(\theta^\star,\psi^\star)}$ is a countable union of zero-measure sets, so it has zero measure. This set is the set of initializations whose associated solutions converge to a strict saddle point, which proves the theorem.
	\end{proof}

	\subsection{Further Analysis and Behavior Around Minimizers}\label{sec::hartman}

		 We keep the same notations as in the proof of Theorem~\ref{thm::MainResDIN}. We saw that to each $\lambda_p<0$ corresponds a positive eigenvalue of $DG(\theta^\star,\psi^\star)$. We now study the case $\lambda_p>0$. The two eigenvalues are the roots of \eqref{eq::chiMp}, whose discriminant is $\Delta_{M_p}=(\alpha+\beta\lambda_p)^2-4\lambda_p$. We denote $\lmin = \frac{(\sqrt{1-\alpha\beta}-1)^2}{\beta^2}$ and $\lmax= \frac{(\sqrt{1-\alpha\beta}+1)^2}{\beta^2}$, the sign of $\Delta_{M_p}$ is then given in the following lemma, proved later in Appendix~\ref{sec::teclemmas}.
		\begin{lemma}\label{lem::signofDelta}
			Let $\alpha\geq 0$, $\beta>0$ and $\lambda\in\R$. The quantity $(\alpha+\beta\lambda)^2-4\lambda$ is non-positive if and only if $\alpha\beta\leq1$ and $\lambda\in \left[\lmin, \lmax\right]$.
		\end{lemma}
		Then, either $\Delta_{M_p}\geq 0$, in that case the eigenvalues are
		$
		\sigma_{p,-} = -\frac{\alpha+\beta\lambda_p}{2} - \frac{\sqrt{\Delta_{M_p}}}{2} \quad\text{and}\quad  \sigma_{p,+} = -\frac{\alpha+\beta\lambda_p}{2} + \frac{\sqrt{\Delta_{M_p}}}{2},
		$
		or $\Delta_{M_p}< 0$ and the eigenvalues are complex valued:
		\begin{equation}\label{eq::negRoots}
			\sigma_{p,-} = -\frac{\alpha+\beta\lambda_p}{2} - i\frac{\sqrt{-\Delta_{M_p}}}{2} \quad\text{and}\quad \sigma_{p,+} = -\frac{\alpha+\beta\lambda_p}{2} + i\frac{\sqrt{-\Delta_{M_p}}}{2}.
		\end{equation}
		We thus deduce the following result when $\lambda_p>0$ (proved in Appendix~\ref{sec::teclemmas}).
		\begin{lemma}\label{lem::lambdappositive}
			With the notations of the proof of Theorem~\ref{thm::MainResDIN}, $\lambda_p>0 \implies \Re(\sigma_{p,+})<0$ and $\Re(\sigma_{p,-})<0$, where $\Re$ denotes the real-part operator.
		\end{lemma}

		In view of Lemma~\ref{lem::lambdappositive} and Theorem~\ref{thm::stabman}, we see that the stable manifolds associated to local minimizers of $\J$ with non-singular Hessian (\textit{i.e.}, where all $\lambda_p>0$), do not have measure zero. This holds \text{e.g.}, for all the local minima of a twice differentiable Morse functions.
		Furthermore, remark that when $\lambda_p>0$ (so in particular around minimizers), $DG(\theta^\star,\psi^\star)$ may have complex eigenvalues (Lemma~\ref{lem::signofDelta}), given in \eqref{eq::negRoots}.
		The existence of complex eigenvalues affects the behavior of the solutions of  \eqref{eq::DIN} around these minimizers as we show below.

	\subsubsection{The Hartman-Grobman Theorem}
	To this aim, we introduce the Hartman-Grobman theorem.
	\begin{theorem}[Hartman–Grobman \citep{perko2013differential}]\label{thm::Hartman}
		Consider the following dynamical system,
		\begin{equation}\label{eq::gensystem}
			\dt\Theta(t) = F(\Theta(t)),\quad t\in\R,
		\end{equation}
		where $\Theta:\R\to\R^{2P}$, $F:\R^{2P}\to\R^{2P}$ is $C^1$ and denote by $DF$ the Jacobian matrix of $F$. Assume that there exists $\Theta^\star\in\R^{2P}$ such that $F(\Theta^\star) = 0$ and $DF(\Theta^\star)$ has only eigenvalues with non-zero real part. Then, there exists a neighborhood $\Om_{\Theta^\star}$ of $\Theta^\star$ and a homeomorphism $H$ (a bijective continuous function whose inverse is continuous) such that, for any $\Theta_0\in \Om_{\Theta^\star}$, if $\Theta$ is a solution of \eqref{eq::gensystem} with $\Theta(0)=\Theta_0$, there exists an open interval of time $\mathsf{T}\subset \R$ containing $0$ such that the function $\Phi = H\circ\Theta$ is the solution of
		\begin{equation}\label{eq::linearized}
			\dt\Phi(t) = DF(\Theta^\star)\Phi(t),\quad t\in \mathsf{T},
		\end{equation}
		with initial condition $\Phi(0) = H(\Theta_0)$. The homeomorphism $H$ preserves the parametrization by time (it does not reverse time).
	\end{theorem}
	This theorem essentially states that around stationary point $\Theta^\star$, if  $DF(\Theta^\star)$ has no purely imaginary eigenvalues, then the qualitative behavior of the solutions of \eqref{eq::gensystem} is similar to that of the linearized system \eqref{eq::linearized}.

	\subsubsection{Application to DIN}\label{sec::HGonDIN}
	\paragraph{Study of the linearized dynamics.} Let $(\theta^\star,\psi^\star)\in\ma $ such that $\theta^\star$ is a local minimizer of $\J$ and such that $\nabla^2\J(\theta^\star)$ has only positive eigenvalues (holds \textit{e.g.}, for Morse functions). Without loss of generality (\textit{i.e.}, up to a translation), assume that $(\theta^\star,\psi^\star)=(0,0)$.
	We first focus on the linearized system: fix an initialization $(\tildet_0,\tildep_0)\in\R^P\times\R^P$ and consider the differential equation
	\begin{equation}\label{eq::DINlinearized}
		\dt{}\begin{pmatrix}\tildet(t)\\\tildep(t)\end{pmatrix} = DG(0,0) \begin{pmatrix}\tildet(t)\\\tildep(t)\end{pmatrix},\quad t \in \R.
	\end{equation}
	This is a linear first-order ODE whose solution is  $\begin{pmatrix}\tildet(t)\\\tildep(t)\end{pmatrix} = Qe^{tD}Q^{-1}\begin{pmatrix}\tildet_0\\ \tildep_0\end{pmatrix}$, for all $t$, where $Q$ is a square matrix of size $2P\times 2P$ and $D$ is diagonal and contains the eigenvalues of $DG(0,0)$ such that $ DG(0,0) = QDQ^{-1}$.

	Since $\nabla^2\J(0)$ is assumed to have only positive eigenvalues, Lemma~\ref{lem::lambdappositive} states that all the eigenvalues of $DG(0,0)$ have negative real part. So, if all these eigenvalues are real, we see that $(\tildet(t),\tildep(t))$ is a sum of exponential functions decaying exponentially fast to the stationary point $(0,0)$ as $t\to + \infty$.
	However, if there exists an eigenvalue $\lambda_p$ of $\nabla^2\J(0)$ such that the corresponding $\Delta_{M_p}$ is negative. Then, $DG(0,0)$ has two complex eigenvalues given in \eqref{eq::negRoots} and the corresponding coordinates of the aforementioned matrix $e^{tD}$ take the form $e^{\frac{-(\alpha+\beta\lambda_p)}{2} t}\left(\cos(\sqrt{-\Delta_{M_p}} t) \pm i \sin(\sqrt{-\Delta_{M_p}} t)\right)$. Therefore $(\tildet(t),\tildep(t))$ still converges exponentially fast to the point $(0,0)$, but spirals around it due to the cosine and sine terms.
	Note that $\sqrt{-\Delta_{M_p}}$ is a decreasing function of $\beta$ (when $\Delta_{M_p}<0$), so increasing the parameter $\beta$ reduces the oscillations.

	\paragraph{Application of the theorem.} We now show that \eqref{eq::DIN} also exhibits the behavior discussed above. Since all the eigenvalues of $DG(0,0)$ have negative real part we can apply Theorem~\ref{thm::Hartman}: there exists a neighborhood $\Om$ of $(0,0)$ and a homeomorphism $H$ for which Theorem~\ref{thm::Hartman} holds. Consider a solution $(\theta,\psi)$ of \eqref{eq::systemG} with initial condition $(\theta_0,\psi_0)\in\Om$. Then, Theorem~\ref{thm::Hartman} states that $(\tildet,\tildep) = H(\theta,\psi)$ is solution to \eqref{eq::DINlinearized} with initial condition $H(\theta_0,\psi_0)$.
	Since $H$ is a homeomorphism, $(\theta,\psi)$ exhibits a similar spiraling behavior\footnote{Additionally $H$ preserves the parametrization by time (\textit{i.e.}, the orientation of oriented curves is preserved see \citep[Chapter 2.8, Definition 1]{perko2013differential}).} as that of $(\tildet,\tildep)$. So according to Lemma~\ref{lem::signofDelta}, for any $(\theta_0,\psi_0) \in\Om$, the corresponding solution $(\theta,\psi)$ spirals while converging to $(0,0)$ if and only if $\alpha\beta\leq 1$ and there exists eigenvalues of $\nabla^2\J(0)$ belonging to $]\lmin,\lmax[$.

	Finally, the discussion is generalized to any $(\theta_0,\psi_0)\in\R^P\times\R^P$: if the corresponding solution $(\theta,\psi)$ converges to $(0,0)$, then there exists $t_0\geq 0$ such that $(\theta(t),\psi(t))\in\Om$ for all $t\geq t_0$ and the arguments above hold after $t_0$.

	\subsubsection{Numerical Illustration of the Spiraling Phenomenon}
	We illustrate the spiraling phenomenon on a 2D convex quadratic function defined by $\J\colon(\theta_1,\theta_2)\in\R^2\mapsto\theta_1^2 + 2\theta_2^2$. It has a unique minimizer $(0,0)$ and a constant diagonal Hessian whose eigenvalues  are $\{2,4\}$. We take initial condition $(1,1)$ and approximate the solution of \eqref{eq::DIN} via the algorithm INNA (with small step-sizes), derived from \eqref{eq::DIN} and presented in next section.
	We consider two choices of parameters: $(\alpha,\beta)=(2,1)$ and $(\alpha,\beta)=(2,0.1)$, to illustrate the cases $\alpha\beta>1$ and $\alpha\beta<1$, respectively. For this last choice, the range of eigenvalues for which we should observe spirals is approximately $[1,359]$ (see Lemma~\ref{lem::signofDelta}), which includes $\{2,4\}$.

	The expected behavior is observed on the left of Figure~\ref{fig::spiral}: the trajectory spirals around $(0,0)$ when $\alpha\beta<1$ (red curve), and does not when $\alpha\beta>1$ (orange curve). Remark that when zooming very close to $(0,0)$, the oscillating behavior remains observable. However this qualitative results say nothing about the speed of convergence, as evidenced on the right of Figure~\ref{fig::spiral}. Spiraling trajectories may converge faster than non-spiraling ones. This is because the Hartman-Grobman theorem connects the solutions of \eqref{eq::DIN} and those of its linearized approximation through a mapping which is homeomorphic (hence continuous) but not necessarily differentiable, so the speed of convergence need not be preserved.
	Interestingly, rates of convergence for \eqref{eq::DIN} can be derived using the KL property (see \cite{castera2019inertial}), which is another reparametrization of the function around critical points, that allows deriving quantitative results instead of qualitative ones, as provided here.

	We also empirically investigate the case where $\alpha$ in \eqref{eq::DIN} is replaced by $\alpha/t$ due to its link with Nesterov's method  \citep{attouch2016fast}.
	In that case there is eventually a time after which $\alpha(t)\beta\leq 1$. Although this is out of the scope\footnote{We could consider non-autonomous ODEs \citep{palmer1973}, but we do not for the sake of simplicity.} of Theorem~\ref{thm::Hartman}, the spiraling phenomenon occurs. We can see it on Figure~\ref{fig::spiral}: for $(\alpha(t),\beta)=(2/t,0.1)$ (blue curve) the spirals are very large and the algorithm is much slower than it was for fixed $\alpha$. However using a larger $\beta$ (green curve), spirals are reduced (although still noticeable), resulting in faster convergence.
	\begin{figure}[t]
		\centering
		\begin{minipage}{0.5\linewidth}
			\centering
			\includegraphics[width=\linewidth]{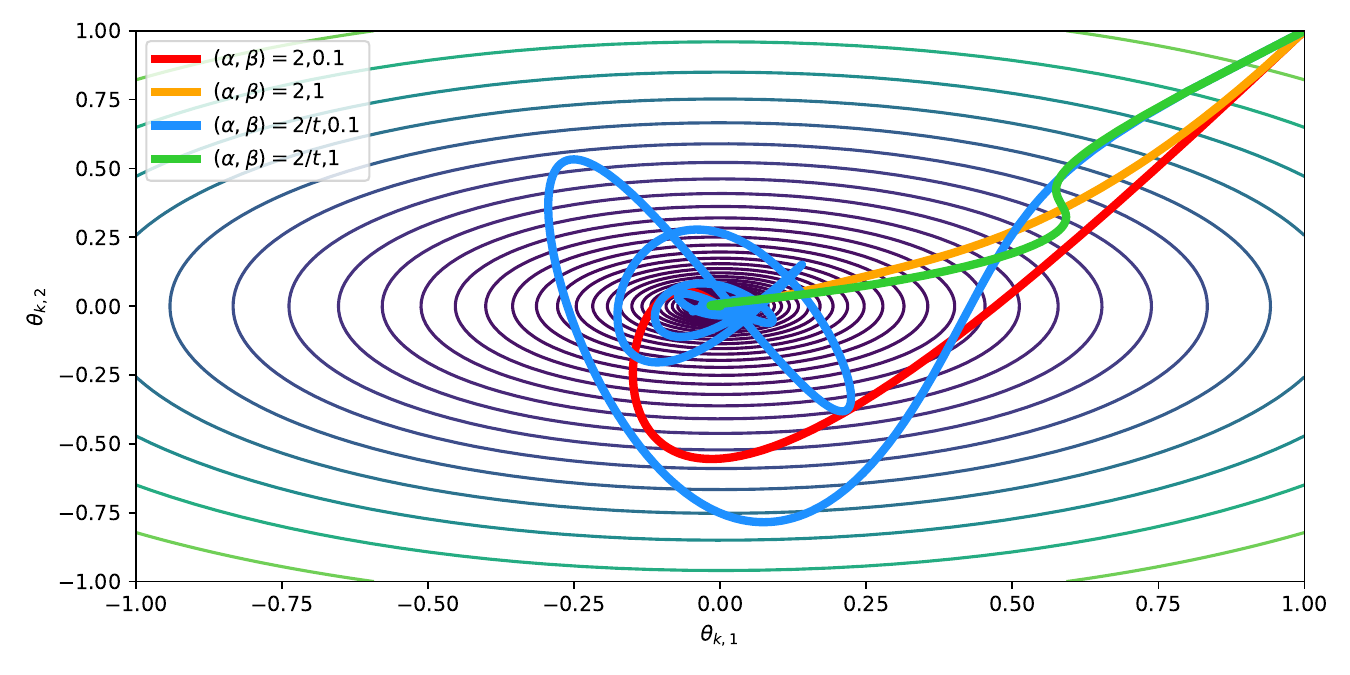}
			\\
			\begin{minipage}{0.49\linewidth}
				\centering
				\tiny Zoom x$10^2$\\
				\includegraphics[width=\linewidth]{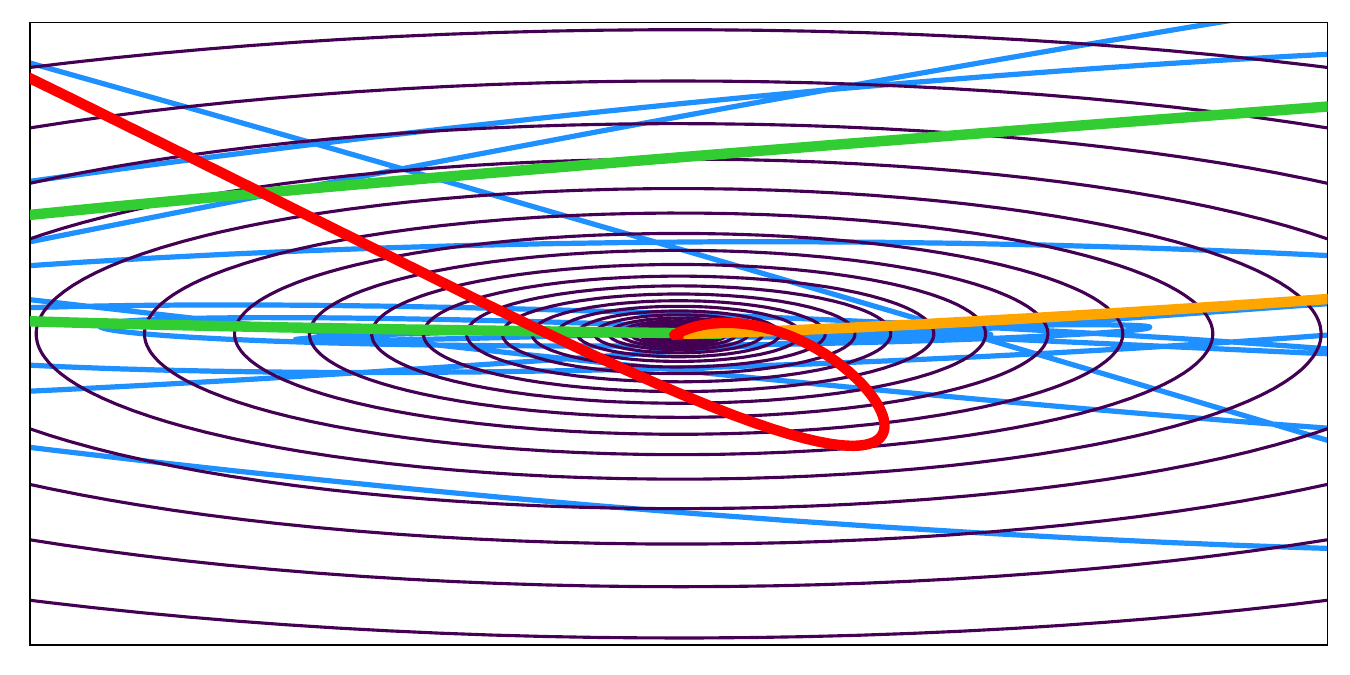}
			\end{minipage}
			\begin{minipage}{0.49\linewidth}
				\centering
				\tiny Zoom x$10^4$\\
				\includegraphics[width=\linewidth]{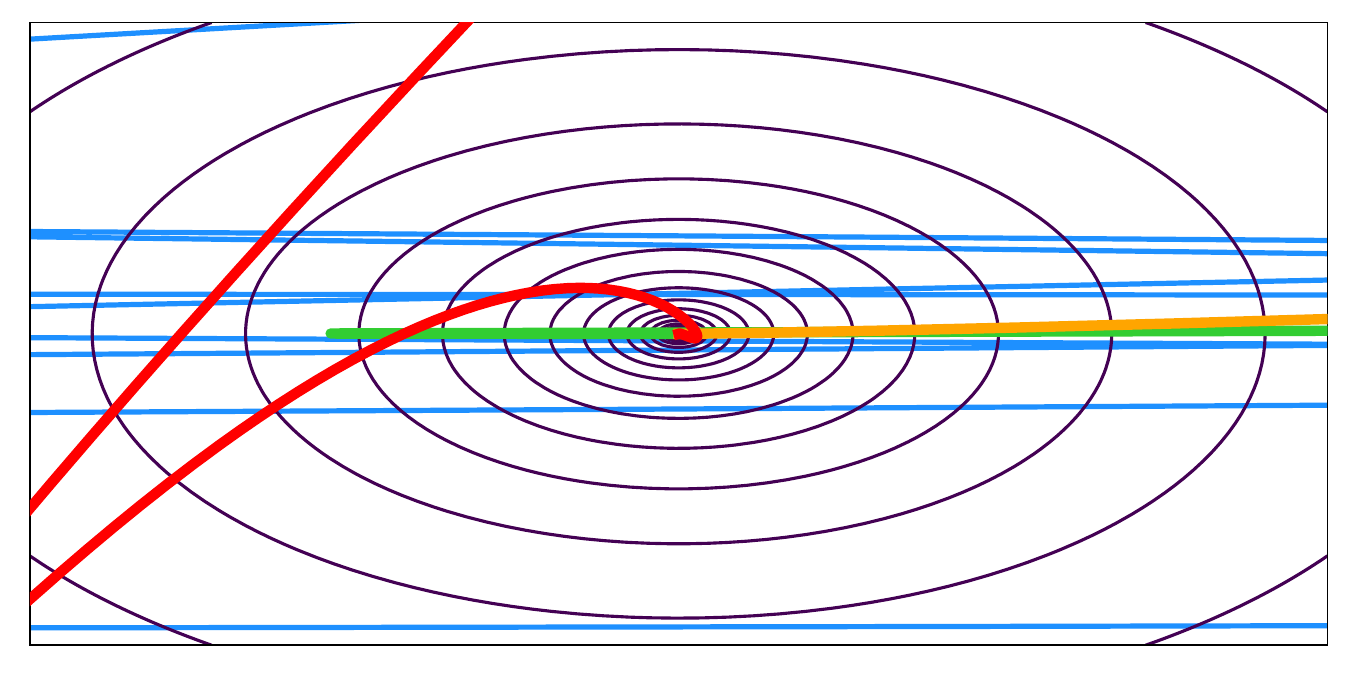}
			\end{minipage}
		\end{minipage}
		\begin{minipage}{0.2\linewidth}
			\includegraphics[width=\linewidth]{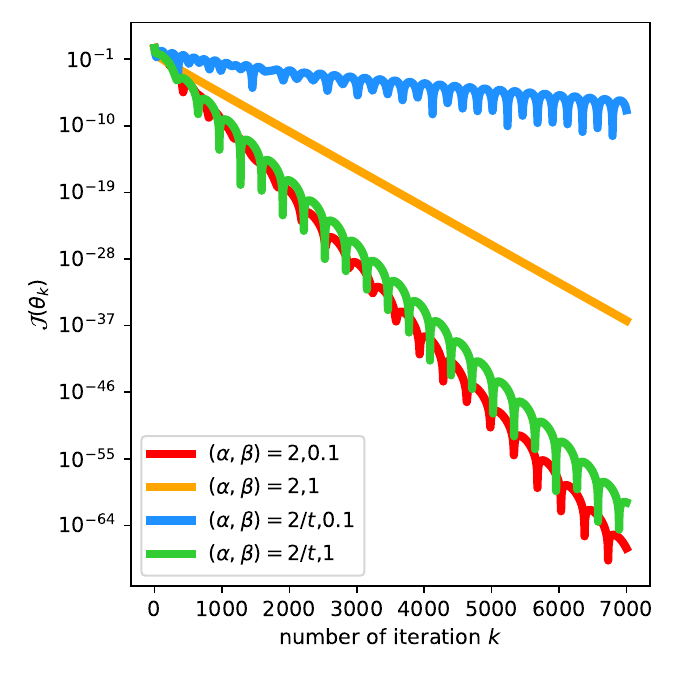}
			\\
			\includegraphics[width=\linewidth]{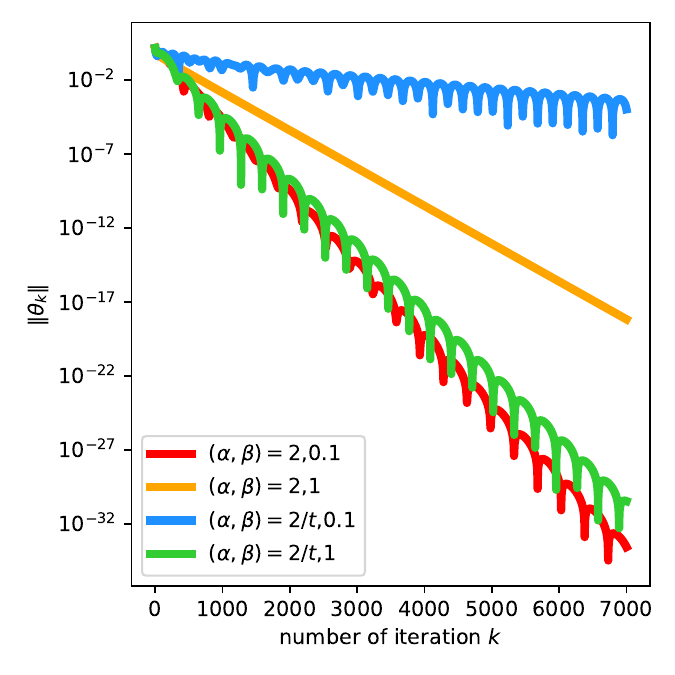}
		\end{minipage}
		\caption{\label{fig::spiral}Illustration of the spiral phenomenon. Left: trajectory on the landscape of $\J$ with two zooms on bottom-left figures. Right: value and distance to the minimizer against iterations.}
	\end{figure}

	\section{Asymptotic Behavior of INNA}\label{sec::INNA}

	We now turn our attention to the asymptotic behavior of the algorithm INNA \cite{castera2019inertial}.  INNA was originally designed for non-smooth and stochastic applications, yet, here we consider the case where $\J$ is still non-convex but $C^2$ and the algorithm is non-stochastic, so we can use fixed step-sizes. Fix $\alpha\geq 0$, $\beta>0$ and a step-size $\gamma$ to be the hyper-parameters of INNA, the algorithm reads:
	\begin{equation}\label{eq::Inna}
		\begin{cases}
			\theta_{k+1} &= \theta_k + \gamma\left[-(\alpha-\frac{1}{\beta})\theta_k  -\frac{1}{\beta}\psi_k -\beta \nabla\J(\theta_k) \right]\\
			\psi_{k+1} &= \psi_k + \gamma\left[-(\alpha-\frac{1}{\beta})\theta_k  -\frac{1}{\beta}\psi_k \right]
		\end{cases}.
	\end{equation}
	This algorithm is obtained via an explicit Euler discretization of \eqref{eq::DIN}, which is why they have shared properties: DIN and INNA have the same stationary points $\ma$ (see Appendix~\ref{sec::AppStabINNA}), and we will show that INNA is likely to avoid $\ma_{<0}$.
	\subsection{INNA Generically Avoids Strict Saddles}
	The major difficulty for INNA compared to DIN, is that the step-size $\gamma$ must be chosen carefully. To this aim, we need the following assumption.
	\begin{assumption}\label{ass::gradLipschitz}
		There exists $L>0$ such that $\nabla\J$ is $L$-Lipschitz continuous on $\R^P$, that is for any $\theta_1,\theta_2\in\R^P$, $\Vert  \nabla\J(\theta_1)-\nabla\J(\theta_2) \Vert \leq L \Vert \theta_1-\theta_2\Vert.$
	\end{assumption}
	Assumption~\ref{ass::gradLipschitz} implies that the eigenvalues of $\nabla^2\J$ are bounded by $L$ on $\R^P$.
	\begin{theorem}\label{thm::MainResINNA}
		Under Assumption~\ref{ass::gradLipschitz}, assume $\alpha>0$ and $\gamma>0$ and consider the following conditions:\\
		(i)\ $\gamma<\frac{\alpha+\beta L- \sqrt{(\alpha+\beta L)^2-4L}}{2L}$,
		\ and,\
		(ii)\
		$\gamma<\frac{-\alpha+\beta L + \sqrt{(\alpha-\beta L)^2+4L}}{2L}$.
		\\
		Then, for almost any initialization, INNA does not converge to a point in $\ma_{<0}$ if $\alpha\beta>1$ and (i) holds. When $\alpha\beta\leq1$, the same is true  if (ii) holds and either $L\in[\lmin,\lmax]$ or (i) holds.
	\end{theorem}
	Remark that unlike for DIN, we do not assume that $\J$ is a Morse function. The proof is postponed to Appendix~\ref{sec::AppStabINNA} and relies on similar arguments  as those used for Theorem~\ref{thm::MainResDIN}. It requires however a different stable manifold theorem: for a function $F$, denote $F^k=\underbrace{F\circ \ldots\circ F}_{k\ \text{compisitions}}$, the theorem is the following.
		\begin{theorem}[{\cite[III.7]{shub2013global}}]\label{thm::DISCstabman}
			Let $\Theta^\star\in\R^{2P}$ be a fixed point for the $C^1$ local diffeomorphism $F:\mathsf{U}\to\R^{2P}$ where $\mathsf{U}\subset\R^{2P}$ is a neighborhood of $\Theta^\star$. Let  $\mathsf{E}^{sc}_{\Theta^\star}$ be the linear subspace spanned by the eigenvalues of $DF(\Theta^\star)$ with magnitude less than one.
			There exists a neighborhood $\Om_{\Theta^\star}$ of $\Theta^\star$ and a $C^1$ manifold $\mathsf{W}^{sc}_{\Theta^\star}$ tangent to $\mathsf{E}^{sc}_{\Theta^\star}$ at $\Theta^\star$---whose dimension is the number of eigenvalues of $DF(\Theta^\star)$ with magnitude less than one---such that, for $\Theta_0\in\R^{2P}$,
			\renewcommand{\theenumi}{(\roman{enumi})}%
			\begin{enumerate}
				\item  If $\Theta_0\in\mathsf{W}^{sc}_{\Theta^\star}$ and $F(\Theta_0)\in \Om_{\Theta^\star}$ then $ F(\Theta_0)\in\mathsf{W}^{sc}_{\Theta^\star}$ (Invariance).
				\item If $\forall k\in\N_{>0},\ F^k(\Theta_0)\in \Om_{\Theta^\star}$, then  $\Theta_0\in \mathsf{W}^{sc}_{\Theta^\star}$.
			\end{enumerate}%
		\end{theorem}
	Theorem~\ref{thm::MainResINNA} is proved by introducing a mapping $G$ such that INNA  reads $(\theta_{k+1},\psi_{k+1})= G(\theta_{k},\psi_{k})$, and by applying Theorem~\ref{thm::DISCstabman} around the fixed points of $G$. The proof involves however significant technical difficulties related to the step-size $\gamma$ (see Appendix~\ref{sec::AppStabINNA}).
	Theorem~\ref{thm::MainResINNA} is non-trivial only if INNA converges, so we provide sufficient conditions so that this is the case.
	\begin{theorem}\label{thm::convergenceofINNA}
		Assume that $\alpha>0$, under Assumption~\ref{ass::gradLipschitz}, if $\gamma$ is such that,
		\begin{equation}\label{eq::condgammaCV2}
			0<\gamma<\min\left(\frac{2\alpha}{(1+\alpha\beta)L+\alpha^2},\, \frac{1}{\alpha} + \beta,\, 2\beta\right),
		\end{equation}
		then any sequence $(\theta_k,\psi_k)_{k\in\N}$ generated by INNA is such that
		$(\J(\theta_k))_{k\in\N}$ converges and $\lim\limits_{k\to+\infty}\Vert\nabla\J(\theta_k)\Vert^2=0$.
		If in addition $(\theta_k)_{k\in\N}$ is bounded and $\J$ has isolated critical points, then $(\theta_k)_{k\in\N}$ converges to a critical point of $\J$.

	\end{theorem}
		The proof is postponed to Appendix~\ref{sec::AppCV}.
		We make the following remarks.
	\begin{remark}[Comments on Theorem~\ref{thm::convergenceofINNA}]
		\item[--] In \eqref{eq::condgammaCV2}, only the condition $\gamma<2\alpha/\left((1+\alpha\beta)L+\alpha^2\right)$ depends on $\J$ (through $L$), and it resembles the condition $\gamma<2/L$ used for GD \citep[Proposition~2.3.2]{bertsekas1998nonlinear}.
		\item[--] From the first part of Theorem~\ref{thm::convergenceofINNA} any converging sub-sequence of $(\theta_k)_{k\in\N}$ converges to critical point. The boundedness assumption then ensures the convergence of the full sequence. It holds for example when $\J$ is coercive. A more general result with relaxed boundedness could be stated, see \cite[Theorem~2.1]{truong2018backtracking}.
		\item[--] The assumption of isolated critical points might be removed  when $\J$ has the KL property, see \textit{e.g.}, \cite{attouch2013convergence,ochs2018local,laszlo2021convergence}.
		\item[--] By combining Theorem~\ref{thm::MainResINNA} and~\ref{thm::convergenceofINNA}, we could formulate a practical corollary for INNA, analogous to Corollary~\ref{cor::dincor}.
	\end{remark}

	\subsection{Numerical Illustration}\label{sec::numexp}
	\begin{figure}[t]
		\centering
		\begin{minipage}{0.38\linewidth}
			\centering
			\includegraphics[width=\linewidth]{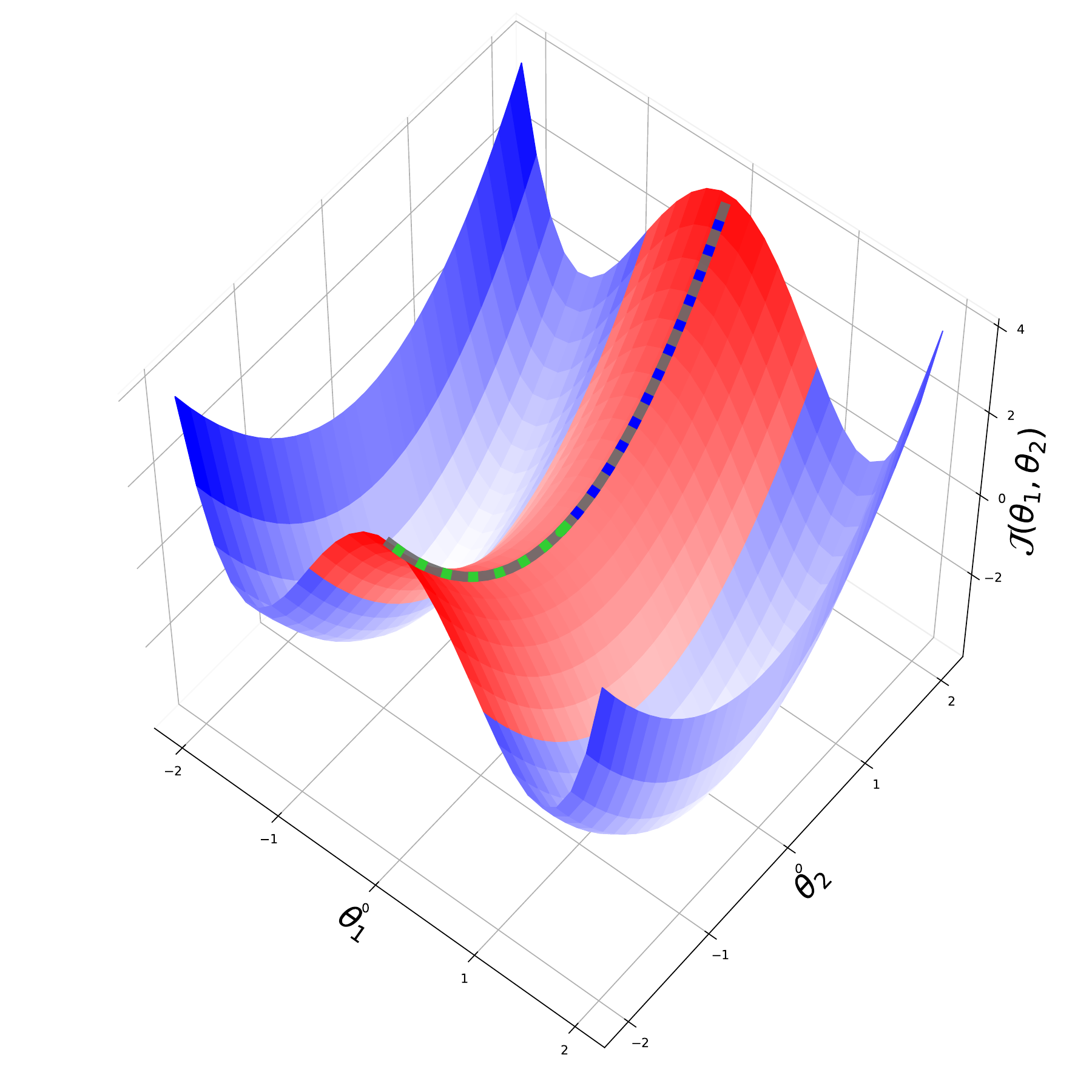}
		\end{minipage}
		\begin{minipage}{0.38\linewidth}
			\centering
			\includegraphics[width=\linewidth]{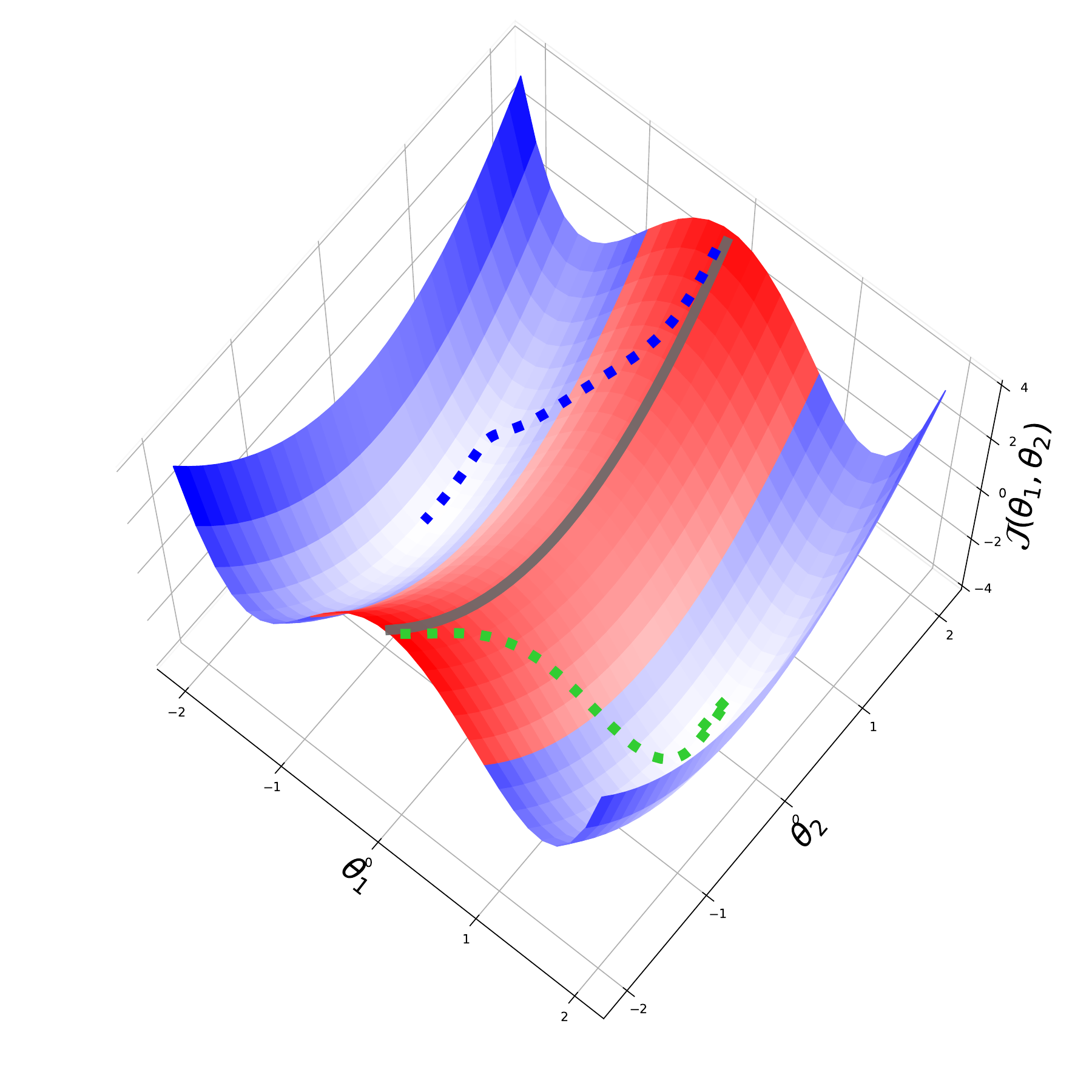}
		\end{minipage}
		\begin{minipage}{0.6\linewidth}
			\includegraphics[width=\linewidth]{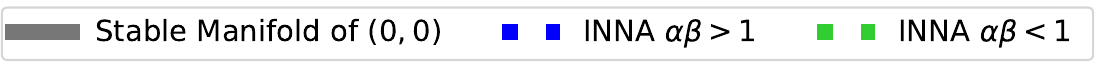}
		\end{minipage}
		\caption{\label{fig::escaping} Evolution of the iterates of INNA on the landscape of the 2D function $\J$ of Section~\ref{sec::numexp} for two choices of $(\alpha,\beta)$. Red and blue surfaces represent locally concave and convex parts of $\J$ respectively. Left figure corresponds to initializations on the stable manifold of $(0,0)$, which yield convergence to $(0,0)$. Right figure represents initializations outside the manifold and convergence to local minimizers.}
	\end{figure}

	We finish the study of INNA by illustrating Theorem~\ref{thm::MainResINNA} on a 2D example. We consider the non-convex $C^2(\R^2)$ function $ \J(\theta_1,\theta_2)=\theta_1^4 - 4\theta_1^2 + \theta_2^2$. It has two minimizers $(-\sqrt{2},0)$ and $(\sqrt{2},0)$ and one strict saddle $(0,0)$. The landscape of $\J$ and the results are displayed on Figure~\ref{fig::escaping}. INNA converges to the strict saddle point $(0,0)$ only when initialized on the manifold $\theta_2\in\R\mapsto (0,\theta_2)$, which has indeed zero measure  (left figure). When initialized anywhere else, it avoids the strict saddle (right figure).
	In addition to this illustration, we ran INNA for $1000$ random Gaussian initializations sampled from $\mathcal{N}_2(0,10^{-24})$, hence extremely close to the saddle point $(0,0)$. In all cases, INNA escaped the saddle $(0,0)$ and converged\footnote{Since $\nabla \J$ is not globally Lipschitz, a local Lipschitz constant $\hat L=50$ was used to locally satisfy the assumptions of Theorems~\ref{thm::MainResINNA} and~\ref{thm::convergenceofINNA}. Convergence was also empirically checked.} to one of the two minima.
	We then did the same for $1000$ random initializations on the stable manifold of $(0,0)$ and INNA always converged to $(0,0)$. The results and code to reproduce these experiments are available at \url{https://github.com/camcastera/INNAavoidsSaddles}.

	\section{Conclusion}
	We showed that DIN and INNA are likely to avoid strict saddle points despite their second-order nature, and even for functions with non-isolated critical points in the case of INNA. This makes them relevant to tackle non-convex minimization problems. We also provided new convergence results for INNA with fixed step-size and a qualitative analysis of the influence of $\alpha$ and $\beta$ on the solutions of DIN around minimizers.
	As for future work, we could extend the results to INNA with varying step-sizes $\gamma_k$. We may do so in two cases where recent progress was made for GD: when $\gamma_k$ is chosen via line-search \citep{truong2019convergence}, and when $\gamma_k$ asymptotically vanishes \citep{panageas2019first}. Finally, one may tackle the stochastic case where we intuitively expect that stochasticity actually helps INNA escaping saddle points as it is the case for stochastic GD \citep{mertikopoulos2020almost}.

	\FloatBarrier

	\section*{Acknowledgment}
	The author acknowledges the support of the European Research Council (ERC FACTORY-CoG-6681839) and the Air Force Office of Scientific Research (FA9550-18-1-0226). The author deeply thanks J\'er\^ome Bolte, C\'edric F\'evotte, and Edouard Pauwels for their valuable comments and the anonymous reviewers for their suggestions which led to significant improvements, such as tackling non-isolated critical points. The numerical experiments were made with the following libraries:  \citep{rossum1995python,walt2011numpy,hunter2007matplotlib}.

	\newpage
	\appendix
	\begin{flushleft}
		{\LARGE \textbf{Appendices}}
	\end{flushleft}


	\section{Permutation Matrices}\label{sec::Permut}
	\newcommand{\modulo}{\mathrm{mod}}
	We specify the permutations matrices used to obtain the block diagonalization in \eqref{eq::blockdiagMat}. Denote by $\modulo$ the modulo operator. We can choose the permutation matrix $U\in\R^{2P\times 2P}$ as the matrix whose coefficients are all zero except the following, for all $p\in\{1,\ldots,P\}$,
	\begin{equation}
		\label{eq::permutationU}
			\text{$P$ odd:}\
			\begin{cases}
				U_{P-p+1,p} &=1-\modulo(p,2)\\
				U_{P+p,2P-p+1}&=\modulo(p,2)\\
				U_{p,2P-p}& = \modulo(p,2)
			\end{cases},
			\text{ $P$ even:}\
			\begin{cases}
				U_{p,p} &=\modulo(p,2)\\
				U_{P+p,P+p}&=1-\modulo(p,2)\\
				U_{P+p,p} &= \modulo(p,2)\\
				U_{p,P+p} &= \modulo(p,2)
			\end{cases}.
	\end{equation}

	\section{Proof of Theorem~\ref{thm::MainResINNA}}\label{sec::AppStabINNA}

	We consider functions with possibly uncountably-many critical points, this yields additional difficulties, which we overcome using the following result as done in \cite{panageas2017}.
	\begin{lemma}[Lindelőf \cite{kelley2017general}]\label{lem::lindelof}
		For every open cover there is a countable sub-cover.
	\end{lemma}

		The proof of Theorem~\ref{thm::MainResINNA} follows similar steps as that of Theorem~\ref{thm::stabman}, so we omit some details and use the notations of Section~\ref{sec::stabmanif}. First, for any $(\theta,\psi)\in\R^P\times\R^P$, we redefine $G$:
	\begin{equation}\label{eq::InnaG}
		G\begin{pmatrix}
			\theta\\\psi
		\end{pmatrix} = \begin{pmatrix}
			\theta + \gamma\left[-(\alpha-\frac{1}{\beta})\theta  -\frac{1}{\beta}\psi -\beta \nabla\J(\theta) \right]\\
			\psi + \gamma\left[-(\alpha-\frac{1}{\beta})\theta  -\frac{1}{\beta}\psi \right]
		\end{pmatrix},
	\end{equation}
	so that iterations $k\in\N$ of INNA read $(\theta_{k+1},\psi_{k+1}) = G(\theta_{k},\psi_{k})$.
	Remark that the set of fixed points of $G$ is $\ma$, the stationary points of \eqref{eq::DIN}, indeed, for any $(\theta,\psi)\in\R^P\times\R^P$,
	\begin{equation*}
		G(\theta,\psi) = (\theta,\psi) \iff
		\begin{cases}
			-(\alpha-\frac{1}{\beta})\theta  -\frac{1}{\beta}\psi -\beta \nabla\J(\theta) = 0\\
			-(\alpha-\frac{1}{\beta})\theta  -\frac{1}{\beta}\psi = 0
		\end{cases} \iff
		\begin{cases}
			\nabla\J(\theta) = 0\\
			\psi = (1-\alpha\beta)\theta
		\end{cases}.
	\end{equation*}
	Since $G$ is $C^1$ on $\R^{P}\times\R^P$, the Jacobian matrix of $G$ (displayed by block) reads,
	\begin{equation*}
		DG(\theta,\psi) = \begin{pmatrix}
			(1 - \gamma(\alpha-\ovb)) I_P -\gamma\beta\nabla^2\J(\theta) && -\frac{\gamma}{\beta} I_P\\
			-\gamma(\alpha-\ovb) I_P && (1-\frac{\gamma}{\beta})I_P
		\end{pmatrix}.
	\end{equation*}
	We can again block-diagonalize $DG(\theta,\psi)$ (see \eqref{eq::blockdiagMat}), in blocks of the form  (up to symmetric permutations): $M_p = \begin{pmatrix}
		1 - \gamma(\alpha-\ovb) -\gamma\beta\lambda_p & & -\frac{\gamma}{\beta} \\
		-\gamma(\alpha-\ovb) & & 1-\frac{\gamma}{\beta}
	\end{pmatrix}$, where $\lambda_p$ is an eigenvalue of $\nabla^2\J(\theta)$.
	To use the Theorem~\ref{thm::DISCstabman}, we need $G$ to be a local diffeomorphism.
	\begin{theorem}\label{thm:INNAdiffeo}
		Under the same assumptions as that of Theorem~\ref{thm::MainResINNA}, the mapping $G$ defined in \eqref{eq::InnaG} is a local diffeomorphism from $\R^P\times\R^P$ to $\R^P\times\R^P$.
	\end{theorem}
	This result is proved later in Section~\ref{sec::teclemmas}. We can now prove Theorem~\ref{thm::MainResINNA}.

	\begin{proof}[Proof of Theorem~\ref{thm::MainResINNA}]
		Let $\alpha$, $\beta$ and $\gamma$ such that the assumptions of the theorem hold and let $G$ defined in \eqref{eq::InnaG} with these parameters. By  Theorem~\ref{thm:INNAdiffeo}, $G$ is a local diffeomorphism.
		Let $(\theta^\star,\psi^\star)\in\ma_{<0}$, to use Theorem~\ref{thm::DISCstabman} we study the magnitude of the eigenvalues of $DG(\theta^\star,\psi^\star)$. Let $\lambda_p<0$ be a negative eigenvalue of $\nabla^2\J(\theta^\star)$, using the notations and elements stated in the beginning of this section, the eigenvalues of $M_p$ are the roots of
		\begin{equation*}
			\chi_{M_p}(X) = X^2 - \mathrm{trace}(M_p)X + \det(M_p) =X^2 - (2 - \gamma(\alpha+\beta\lambda_p) )X + 1 - \gamma(\alpha+\beta\lambda_p) + \gamma^2\lambda_p .
		\end{equation*}
		The discriminant of $\chi_{M_p}$ is:
		\begin{equation*}
			\Delta_{M_p} = (2 - \gamma(\alpha+\beta\lambda_p) )^2 - 4 (1 - \gamma(\alpha+\beta\lambda_p) + \gamma^2\lambda_p)
			= \gamma^2\left((\alpha+\beta\lambda_p)^2 - 4\lambda_p\right).
		\end{equation*}
		Remark that Lemma~\ref{lem::signofDelta} gives again the sign of $\Delta_{M_p}$. Thus since $\lambda_p<0$, we necessarily have $\Delta_{M_p}\geq 0$, and can ignore the case $\Delta_{M_p}<0$.
		So $M_p$ has two real eigenvalues,
		\begin{equation*}
			\begin{cases}
				\sigma_{p,+} &= 1 - \frac 1 2 \gamma(\alpha+\beta\lambda_p) + \frac{1}{2} \gamma\sqrt{(\alpha+\beta\lambda_p)^2 - 4\lambda_p}  \\
				\sigma_{p,-} &=  1 - \frac 1 2 \gamma(\alpha+\beta\lambda_p) - \frac{1}{2} \gamma\sqrt{(\alpha+\beta\lambda_p)^2 - 4\lambda_p}
			\end{cases}.
		\end{equation*} Since $\lambda_p<0$, then $\vert\alpha+\beta\lambda_p\vert < \sqrt{(\alpha+\beta\lambda_p)^2 - 4\lambda_p}$, so observe that $\sigma_{p,+}>1$ and $\sigma_{p,-}<1$, so $DG(\theta^\star,\psi^\star)$ has at least one eigenvalue with magnitude larger than one.

		We can now use the stable manifold theorem, we omit some details since the arguments are the same as for the proof of Theorem~\ref{thm::MainResDIN}.
		By Theorem~\ref{thm::DISCstabman}, around each $(\theta^\star,\psi^\star)\in\ma_{<0}$, there exists a neighborhood $\Om_{(\theta^\star,\psi^\star)}$ on which the stable manifold theorem holds.
		Denote by $\mathsf{A}$
		the possibly uncountable union of all these neighborhoods: $\mathsf{A} = \bigcup_{(\theta^\star,\psi^\star)\in\ma_{<0}} \Om_{(\theta^\star,\psi^\star)}$. By Lemma~\ref{lem::lindelof}, there exists a countable sub-cover of this set, \textit{i.e.}, there exists a sequence $(\theta_i^\star,\psi_i^\star)_{i\in\N}$ in $\ma_{<0}$ such that
		\begin{equation}\label{eq::subcover}
			\mathsf{A} = \bigcup_{i\in\N} \Om_{(\theta_i^\star,\psi_i^\star)}.
		\end{equation}
		Let $(\theta^\star,\psi^\star)\in\ma_{<0}$, it need not be an element of $(\theta_i^\star,\psi_i^\star)_{i\in\N}$, but according to \eqref{eq::subcover}, there exists $i\in\N$ such that $(\theta^\star,\psi^\star)\in\Om_{(\theta_i^\star,\psi_i^\star)}$. Let
		an initialization $(\theta_0,\psi_0)$ such that the associated realization $(\theta_k,\psi_k)_{k\in\N}$ of INNA  converges to $(\theta^\star,\psi^\star)$. This means that there exists $k_0\in\N$ such that $\forall k\geq k_0$, $G^k(\theta_0,\psi_0)\in\Om_{(\theta_i^\star,\psi_i^\star)}$ and thus $G^k(\theta_0,\psi_0)\in\mathsf{W}^{sc}_{(\theta_i^\star,\psi_i^\star)}$, where $\mathsf{W}^{sc}_{(\theta_i^\star,\psi_i^\star)}$ is the stable manifold around $(\theta_i^\star,\psi_i^\star)$ as defined in Theorem~\ref{thm::DISCstabman}.
		By Theorem~\ref{thm:INNAdiffeo}, $G$ is a local diffeomorphism, so we can reverse the iterations and obtain,
		$
			(\theta_0,\psi_0) \in \bigcup_{j\in\N}G^{-j}\left(\Om_{(\theta_i^\star,\psi_i^\star)}\cap\mathsf{W}^{sc}_{(\theta_i^\star,\psi_i^\star)}\right)
		$.
		 Since $(\theta_i^\star,\psi_i^\star)\in\ma_{<0}$, we showed that $DG(\theta_i^\star,\psi_i^\star)$ has at least one eigenvalue with magnitude strictly larger than $1$, so by Theorem~\ref{thm::DISCstabman}, $W^{sc}_{(\theta_i^\star,\psi_i^\star)}$ has zero measure. Then, by Theorem~\ref{thm:INNAdiffeo} for all $j\in\N$, $G^{-j}$ is a local diffeomorphism, so the union above has zero measure. Using \eqref{eq::subcover}, the rest of the proof is then similar to the end of that of Theorem~\ref{thm::MainResDIN} since $
		 \bigcup_{i\in\N}\left[\bigcup_{j\in\N}G^{-j}\left(\Om_{(\theta_i^\star,\psi_i^\star)}\cap\mathsf{W}^{sc}_{(\theta_i^\star,\psi_i^\star)}\right)\right]
		 $ is a countable union of zero-measure sets, so it has again measure zero.
	\end{proof}

	\subsection{Missing Proofs}\label{sec::teclemmas}
	We begin by proving the lemmas stated in Section~\ref{sec::hartman}.
	\begin{proof}[Proof of Lemma~\ref{lem::signofDelta}]
		Let $\alpha\geq 0$, and $\beta>0$, the function $h(\lambda)=(\alpha+\beta\lambda)^2-4\lambda = \beta^2\lambda^2 + 2(\alpha\beta-2)\lambda +\alpha^2$ is a second-order polynomial in $\lambda$ whose discriminant is $16(1-\alpha\beta)$. If $\alpha\beta>1$ this discriminant is negative so $h$ is always positive. If $\alpha\beta\leq 1$, then $h$ has two real roots: $\frac{(2-\alpha\beta)}{\beta^2} \pm \frac{2\sqrt{1-\alpha\beta}}{\beta^2}$, which are equal to $\lmin$ and $\lmax$ since $X^2\pm2X+1 = (X\pm1)^2$.
	\end{proof}
	\begin{proof}[Proof of Lemma~\ref{lem::lambdappositive}]
	Assume that $\lambda_p>0$, if $\Delta_{M_p}<0$, then $ 2\Re(\sigma_{p,-}) = 2\Re(\sigma_{p,+}) = -(\alpha+\beta\lambda_p)<0$. If $\Delta_{M_p}\geq 0$ then $\sigma_{p,-}$ and $\sigma_{p,+}$ are real. Remark that $\sigma_{p,-}\sigma_{p,+}=\lambda_p$ so the eigenvalues have the same sign, and $\sigma_{p,-}+\sigma_{p,+}= -(\alpha+\beta\lambda_p)<0,$ so they are negative.
	\end{proof}
	\begin{proof}[Proof of Theorem~\ref{thm:INNAdiffeo}]
		Let $(\theta,\psi)\in\R^P\times\R^P$,
		to prove that $G$ is a local diffeomorphism we prove that $DG(\theta,\psi)$ is invertible and then use the local inversion theorem. Using again the block transformation of $DG(\theta,\psi)$, $\det(DG(\theta,\psi)) = \prod_{p=1}^P \det(M_p)$, where
		\begin{equation}\label{eq::detgamma}
			\det(M_p) = (1 - \gamma(\alpha-\ovb) -\gamma\beta\lambda_p)(1-\frac{\gamma}{\beta}) - \frac{\gamma}{\beta}\gamma(\alpha-\ovb) = 1 - \gamma(\alpha+\beta\lambda_p) + \gamma^2\lambda_p.
		\end{equation}
		We want $\gamma$ such that $\det(M_p)\neq 0$ for any $(\theta,\psi)\in\R^P\times\R^P$, hence for any $\lambda_p\in[-L,L]$ (using Assumption~\ref{ass::gradLipschitz}).
		First, if $\lambda_p=0$, from \eqref{eq::detgamma}, we must take $\gamma\neq 1/\alpha$. Now let $\lambda_p\neq 0$, then \eqref{eq::detgamma} is a second-order polynomial in $\gamma$ with discriminant $ (\alpha+\beta\lambda_p)^2 - 4\lambda_p= \Delta_{M_p}$ already studied Section~\ref{sec::stabmanif} and Lemma~\ref{lem::signofDelta}.
		If $\Delta_{M_p}<0$, then $\det(M_p)$ has no real roots and the choice of $\gamma$ is free. Assume now that $\Delta_{M_p}\geq 0$, there exists two real roots to \eqref{eq::detgamma}:
		\begin{equation}\label{eq::gammaRoots}
				\gamma^{+} =  \frac{(\alpha+\beta\lambda_p)}{2\lambda_p} +  \frac{\sqrt{(\alpha + \beta \lambda_p)^2 - 4\lambda_p}}{2 \lambda_p}
				\ \text{and}\
				\gamma^{-} =  \frac{(\alpha+\beta\lambda_p)}{2\lambda_p} -  \frac{\sqrt{(\alpha + \beta \lambda_p)^2 - 4\lambda_p}}{2 \lambda_p}.
		\end{equation}
		Remark that when $\lambda_p<0$, $\gamma^+<0$ and when $\lambda_p>0$, $0<\gamma^-<\gamma^+$, so in every case we only need to ensure $0<\gamma<\gamma^-$, for every $\lambda_p\in[-L,L]$. So, for every $\lambda\in\R$ for which it is well defined, consider the function $\gamma^{-}(\lambda)=\frac{(\alpha+\beta\lambda)}{2\lambda} -  \frac{\sqrt{(\alpha + \beta \lambda)^2 - 4\lambda}}{2 \lambda}$. When defined, its derivative is $ -\frac{\alpha\sqrt{\left(\alpha+\beta \lambda\right)^2-4\lambda}+\left(2-\alpha\beta\right)\lambda-\alpha^2}{2\lambda^2\sqrt{\left(\alpha+\beta \lambda\right)^2-4\lambda}}$. The denominator is always positive so we study the numerator: $h(\lambda)=-\alpha\sqrt{(\alpha+\beta \lambda)^2-4\lambda}-(2-\alpha\beta)\lambda+\alpha^2 $, and we differentiate it:
		\begin{equation*}
			h'(\lambda)  = -\frac{\alpha(2\beta(\alpha+\beta \lambda)-4)}{2\sqrt{(\alpha+\beta \lambda)^2-4\lambda}}+\alpha\beta-2,
			\quad \text{and}\quad
			h''(\lambda)  = -\frac{4\alpha(\alpha\beta-1)}{((\alpha+\beta \lambda)^2-4\lambda)^{\frac{3}{2}}}.
		\end{equation*}
		This allows deducing the minimal value of $\gamma^-(\lambda)$ in each setting by constructing the tables of variations displayed in Figure~\ref{tab::var}. There, it follows from standard computations that $h'(0)=h(0)=0$, $h(\lmax)\leq 0$ and $\lim_{\lambda\to +\infty} h'(\lambda)=-2$ (when $\alpha\beta\leq 1$), and via L'Hôpital's rule we obtained $\lim_{\lambda\to 0}\gamma^-(\lambda)=1/\alpha$.
		We deduce from the tables that $G$ is a local diffeomorphism if $\gamma<\gamma^-(L)$ when $\alpha\beta>1$ and if $\gamma<\min(\gamma^-(L),\gamma^-(-L))$ when $\alpha\beta\leq 1$ and $L\notin [\lmin,\lmax]$.
		Remark that the condition $\gamma\neq \frac{1}{\alpha}$ is implied in both cases. This proves the theorem.
		\begin{figure}[t]
			\centering
		\begin{tabular}{ c c }
			\tiny If $\alpha\beta>1$ & \tiny  If $\alpha\beta\leq1$\\
			\resizebox{0.35\linewidth}{0.17\linewidth}{%
				\begin{tikzpicture}
					\tkzTabInit[lgt=1.2]{$\lambda$ /1,$h''(\lambda)$ /1,$h'(\lambda)$ /1, $h(\lambda)$/1, $\gamma^-(\lambda)$/1}
					{$-\infty$,$0$,$+\infty$}
					\tkzTabLine{,-,t,-,}
					\tkzTabVar{+/,
						R/, 
						-/$-\infty$}
					\tkzTabIma{1}{3}{2}{$0$} 
					\tkzTabVar{-/,
						+/$0$,
						-/}
					\draw[double style](N24) --(N25); 
					\tkzTabVar{+/,
						R/,
						-/}
					\tkzTabIma{1}{3}{2}{$\frac{1}{\alpha}$} 
				\end{tikzpicture}
			}
			&
			\resizebox{0.55\linewidth}{0.17\linewidth}{%
				\begin{tikzpicture}
					\tkzTabInit[lgt=1.2]{$\lambda$ /1,$h''(\lambda)$ /1,$h'(\lambda)$ /1, $h(\lambda)$/1,$\gamma^-(\lambda)$/1}
					{$-\infty$,$0$,$\lmin$,$\lmax$,$+\infty$}
					\tkzTabLine{,+,t,+,d,h,d,+,}
					\tkzTabVar{-/,
						R/, 
						+DH/,
						D-/,
						+/-2
					}
					\tkzTabIma{1}{3}{2}{$0$} 
					\tkzTabVar{+/,
						-/$0$,
						+DH/,
						+C/$h(\lmax)\leq 0$,
						-/
					}
					\draw[double style](N24) --(N25); 
					\tkzTabVar{-/,
						R/, 
						+DH/,
						D+/,
						-/}
					\tkzTabIma{1}{3}{2}{$\frac{1}{\alpha}$} 
				\end{tikzpicture}
			}
		\end{tabular}
		\caption{Tables of variations for the proof of Theorem~\ref{thm:INNAdiffeo}. The sign of $h''$ allows deducing the variations and signs of $h'$ and $h$ which themselves allow deducing the minima of $\gamma^-$.\label{tab::var}}
		\end{figure}
	\end{proof}

	\section{Proof of Convergence of INNA}\label{sec::AppCV}
	To prove Theorem~\ref{thm::convergenceofINNA}, we will use the following lemma.
	\begin{lemma}[{\cite{avsic1970limit}}]\label{lem::lange}
		If a bounded sequence $(u_k)_{k\in\N}$ in $\R^P$ satisfies,
		$
			\lim\limits_{k\to+\infty} \Vert u_{k+1} - u_k\Vert =0,
		$
		then the set of accumulation points of $(u_k)_{k\in\N}$ is connected. If this set is finite then it reduces to a singleton and $(u_k)_{k\in\N}$ converges.
	\end{lemma}

	\begin{proof}[Proof of Theorem~\ref{thm::MainResINNA}]
		Assume that Assumption~\ref{ass::gradLipschitz} holds and $\alpha>0$. Let $(\theta_0,\psi_0)\in\R^P\times\R^P$, and let $\gamma>0$ such that \eqref{eq::condgammaCV2} holds.
		Let $(\theta_k,\psi_k)_{k\in\N}$ be the sequence generated by INNA initialized at $(\theta_0,\psi_0)$. We first show that the sequence $(\E_k)_{k\in\N}$ defined $\forall k\in\N$ by
		\begin{equation}\label{eq::Ek}
			\E_k = (1+\alpha\beta-\gamma\alpha)\J(\theta_k) + \frac{1}{2}\Vert(\alpha-\frac{1}{\beta})\theta_k  +\frac{1}{\beta}\psi_k\Vert^2
		\end{equation}
		converges. The sequence $(\E_k)_{k\in\N}$ represents an ``energy'' that decreases along the iterations, where the first and second terms in \eqref{eq::Ek} represent ``potential'' and ``kinetic'' energies respectively. This sequence resembles the Lyapunov function of DIN \citep{alvarez2002second,castera2019inertial} but is more involved to derive, as often for algorithms compared to ODEs.
		We use the notations $a=\alpha-1/\beta$, $b=1/\beta$, $\DT_k =\theta_{k+1}-\theta_k$ and $\DP_k =\psi_{k+1}-\psi_k$, for $k\in\N$, so that INNA is rewritten as:
		\begin{equation}\label{eq::SimpINNA}
			\begin{cases}
				\DP_k &=  -\gamma a\theta_k  -\gamma b\psi_k\\
				\DT_k &= \DP_k -\gamma\beta \nabla\J(\theta_k)
			\end{cases}.
		\end{equation}
		Also denote $\mu = 1+\alpha\beta-\gamma\alpha$, where $\mu>0$ since $\gamma < 1/\alpha + \beta$.
		Let $k\in\N$, we will prove $\E_{k+1}-\E_{k}\leq 0$. From Assumption~\ref{ass::gradLipschitz} follows a \emph{descent lemma} (see \cite[Proposition A.24]{bertsekas1998nonlinear}):
		\begin{equation*}
			\mu\J(\theta_{k+1}) - \mu\J(\theta_k) \leq \mu\langle \nabla\J(\theta_k) , \DT_k \rangle + \frac{\mu L}{2} \Vert\DT_k\Vert^2,
		\end{equation*}
		which according to \eqref{eq::SimpINNA}, can equivalently be rewritten as,
		\begin{equation}\label{eq::descentlemma2}
			\mu\J(\theta_{k+1}) - \mu\J(\theta_k) \leq -\mu\langle \frac{\DT_k-\DP_k}{\gamma\beta} , \DT_k \rangle + \frac{\mu L}{2} \Vert\DT_k\Vert^2.
		\end{equation}
		We save this for later and now turn our attention to the other term in $\E_{k+1}-\E_k$,
		\begin{equation*}
			\frac{1}{2}\Vert a\theta_{k+1} + b\psi_{k+1}\Vert^2 - \frac{1}{2}\Vert a\theta_{k} + b\psi_{k}\Vert^2 =  \frac{1}{2}\Vert a\theta_{k} + a\DT_k +  b\psi_{k}+b\DP_k\Vert^2 - \frac{1}{2}\Vert a\theta_{k} + b\psi_{k}\Vert^2  .
		\end{equation*}
		Expanding this and using the fact that $ a\theta_{k} + b\psi_{k} = -\DP_k/\gamma$, we can show that,
		\begin{align}\label{eq::secondterm}
			\begin{split}
				&\frac{1}{2}\Vert a\theta_{k} + a\DT_k +  b\psi_{k}+b\DP_k\Vert^2 - \frac{1}{2}\Vert a\theta_{k} + b\psi_{k}\Vert^2
				\\=& \frac{a^2}{2}\Vert \DT_k\Vert^2 + \frac{b^2}{2} \Vert\DP_k\Vert^2 + ab\langle\DT_k,\DP_k\rangle -\frac{a}{\gamma}\langle\DT_k,\DP_k\rangle  -\frac{b}{\gamma}\Vert\DP_k\Vert^2.
			\end{split}
		\end{align}
		We then use $\Vert\DP_k\Vert^2 = \Vert \DT_k - \DP_k\Vert^2 + \Vert\DT_k\Vert^2 -2\langle\DT_k,\DT_k-\DP_k\rangle$ and $\langle\DT_k,\DP_k\rangle = \Vert\DT_k\Vert^2 -\langle\DT_k,\DT_k-\DP_k\rangle$ in
		\eqref{eq::secondterm} to obtain:
		\begin{multline}\label{eq::secondterm3}
			\frac{1}{2}\Vert a\theta_{k+1} + b\psi_{k+1}\Vert^2 - \frac{1}{2}\Vert a\theta_{k} + b\psi_{k}\Vert^2
			= \left(\frac{a^2}{2}+\frac{b^2}{2}+ ab - \frac{a}{\gamma} - \frac{b}{\gamma}\right)\Vert \DT_k\Vert^2
			\\ + \left(\frac{b^2}{2}-\frac{b}{\gamma}\right)\Vert\DT_k-\DP_k\Vert^2 + \left(-b^2-ab+\frac{a}{\gamma}+\frac{2b}{\gamma}\right)\langle\DT_k,\DT_k-\DP_k\rangle.
		\end{multline}
		We then simplify the factors using the identity $a+b=\alpha$, as well as  $\frac{a^2}{2}+\frac{b^2}{2}+ ab = \frac{1}{2}(a+b)^2 = \frac{\alpha^2}{2}$, and $-b^2-ab = -\alpha/\beta $  to deduce that \eqref{eq::secondterm3} is equal to
		\begin{multline}\label{eq::secondterm4}
			\left(\frac{\alpha^2}{2} - \frac{\alpha}{\gamma}\right)\Vert \DT_k\Vert^2
			+\frac{\gamma-2\beta}{2\gamma\beta^2}\Vert\DT_k-\DP_k\Vert^2
			+\frac{-\gamma\alpha+\alpha\beta+1}{\gamma\beta}\langle\DT_k,\DT_k-\DP_k\rangle.
		\end{multline}
		We can finally combine \eqref{eq::descentlemma2} and \eqref{eq::secondterm4},
		\begin{align}\label{eq::diffEk}
			\begin{split}
				\E_{k+1}-\E_k \leq& \left(\frac{\mu L}{2}+\frac{\alpha^2}{2} - \frac{\alpha}{\gamma}\right)\Vert \DT_k\Vert^2
				+\frac{\gamma-2\beta}{2\gamma\beta^2}\Vert\DT_k-\DP_k\Vert^2
				\\&+ \left(-\frac{\mu}{\gamma\beta}+\frac{1+\alpha\beta-\gamma\alpha}{\gamma\beta}\right)\langle\DT_k,\DT_k-\DP_k\rangle.
			\end{split}
		\end{align}
		Notice that, $\mu = 1+\alpha\beta-\gamma\alpha$ is specifically chosen so that the last term in \eqref{eq::diffEk} vanishes, so,
		\begin{equation}\label{eq::diffEk2}
			\E_{k+1}-\E_k \leq \frac{\gamma\mu L+\gamma\alpha^2 -2\alpha}{2\gamma}\Vert \DT_k\Vert^2
			+\frac{\gamma-2\beta}{2\gamma\beta^2}\Vert\DT_k-\DP_k\Vert^2.
		\end{equation}
		To prove the decrease of $(\E_k)_{k\in\N}$, it remains to justify that both terms in \eqref{eq::diffEk2} are negative. First, the condition $\gamma<2\beta$ in \eqref{eq::condgammaCV2} makes the second term negative. Then,
		\begin{equation*}
			\gamma\mu L+\gamma\alpha^2 -2\alpha <0
			\iff   -\alpha L\gamma^2 + \left(\alpha^2+ (1+\alpha\beta)L\right)\gamma - 2\alpha <0.
		\end{equation*}
		A simpler sufficient condition for this to hold is $\left(\alpha^2+ (1+\alpha\beta)L\right)\gamma - 2\alpha <0$ or equivalently $\gamma< 2\alpha/\left(\alpha^2+ (1+\alpha\beta)L\right)$, which holds from \eqref{eq::condgammaCV2}.
		So the sequence $(\E_k)_{k\in\N}$ is a decreasing. It is also lower-bounded since $\J$ is lower-bounded, so it converges.

		The rest of the proof then relies on exploiting \eqref{eq::diffEk2}. Let $K\in\N$, we sum \eqref{eq::diffEk2}:
		\begin{equation*}
			\sum_{k=0}^{K} \E_{k+1} - \E_{k} \leq \frac{\gamma\mu L+\gamma\alpha^2 -2\alpha}{2\gamma} \sum_{k=0}^{K}  \Vert \DT_k\Vert^2
			+\frac{\gamma-2\beta}{2\gamma\beta^2}\sum_{k=0}^{K} \Vert\DT_k-\DP_k\Vert^2.
		\end{equation*}
		The left-hand side is a telescopic series, and it follows from \eqref{eq::SimpINNA} that $\forall k\in\N$, $\DT_k-\DP_k = -\gamma\beta\nabla\J(\theta_k)$, so denoting $C_1 = -(\gamma\mu L+\gamma\alpha^2 -2\alpha)/2\gamma>0$ and $C_2 = -(\gamma^2-2\gamma\beta)/2>0$,
		\begin{equation*}
			\E_{0}-\E_{K+1} \geq C_1 \sum_{k=0}^{K}  \Vert \DT_k\Vert^2
			C_2\sum_{k=0}^{K} \Vert\nabla\J(\theta_k)\Vert^2.
		\end{equation*}
		Then, $\E_{0} - \E_{K+1}$ is upper bounded since $(\E_k)_{k\in\N}$ converges,
		so $\sum_{k=0}^{K} \Vert \nabla\J(\theta_k)\Vert^2<+\infty$. This implies that $\lim_{k\to +\infty}\Vert \nabla\J(\theta_k)\Vert^2 =0$ and we deduce similarly that $\lim_{k\to +\infty}\Vert \theta_{k+1}-\theta_{k}\Vert^2 =0$.
		Using \eqref{eq::Inna}, we also have,
		\begin{equation}\label{eq::DPconverges}
			\Vert(\alpha-\frac{1}{\beta})\theta_k  +\frac{1}{\beta}\psi_k\Vert^2 = \frac{1}{\gamma^2}\Vert \psi_{k+1}-\psi_{k}\Vert^2\leq \frac{2}{\gamma^2}\Vert\theta_{k+1}-\theta_{k}\Vert^2  +2\beta^2 \Vert\nabla\J(\theta_k)\Vert^2 \xrightarrow[k\to\infty]{}0.
		\end{equation}
		The convergence of $(\E_k)_{k\in\N}$ and \eqref{eq::DPconverges} imply that $(\J(\theta_k))_{k\in\N}$ converges, which proves the first part of the theorem.
		Assume that the critical points are isolated and that the sequence $(\theta_k)_{k\in\N}$ is uniformly bounded on $\R^P$.
		According to Lemma~\ref{lem::lange}, since $(\theta_k)_{k\in\N}$ is bounded and $\lim_{k\to +\infty}\Vert \theta_{k+1}-\theta_{k}\Vert=0$, the set of accumulation points of $(\theta_k)_{k\in\N}$ is connected.
		By continuity of $\nabla \J$, accumulation points of $(\theta_k)_{k\in\N}$ are critical points of $\J$, which are assumed to be isolated. So the set of accumulation points is a singleton and $(\theta_k)_{k\in\N}$ converges to it.
	\end{proof}

	\FloatBarrier
	\bibliographystyle{plainnat}
	\bibliography{biblio.bib}

\end{document}